\documentclass[a4paper,11pt]{amsart}
 
\usepackage{amsmath,amssymb,amsthm}
\input xy
\xyoption{all}
\newtheorem{theorem}{Theorem}[section]

\newtheorem{Theorem}[theorem]{Theorem}
\newtheorem{cor}[theorem]{Corollary}
\newtheorem{lemm}[theorem]{Lemma}
\newtheorem{prop}[theorem]{Proposition}

\newtheorem{conj}[theorem]{Conjecture}
\theoremstyle{definition}
\newtheorem{defi}[theorem]{Definition}
\newtheorem{remk}[theorem]{Remark}
\newtheorem{exam}[theorem]{Example}
\newcommand{\DD}{{\mathcal D}}
\newcommand{\FF}{{\mathcal F}}
\newcommand{\XX}{{\mathcal X}}

\newcommand{\F}{\mathbf{F}}

\newcommand{\T}{\mathbf{T}}
\newcommand{\Q}{\mathbf{Q}}
\newcommand{\RT}{\mathbb{R}\mathbf{T}}
\newcommand{\LT}{\mathbb{L}\mathbf{T}}
\newcommand{\La}{\Lambda}

\newcommand{\add}{\operatorname{add}\nolimits}
\newcommand{\proj}{\operatorname{proj}\nolimits}

\newcommand{\gl}{\operatorname{gl.dim}\nolimits}

\newcommand{\Hom}{\operatorname{Hom}\nolimits}

\newcommand{\Ext}{\operatorname{Ext}\nolimits}
\newcommand{\Tor}{\operatorname{Tor}\nolimits}

\newcommand{\RHom}{\mathbb{R}\strut\kern-.2em\operatorname{Hom}\nolimits}
\newcommand{\Lotimes}{\mathop{\stackrel{\mathbb{L}}{\otimes}}\nolimits}
\renewcommand{\mod}{\operatorname{mod}\nolimits}
\def\dim{\mathop{\mathrm{dim}}\nolimits}

\def\Ker{\mathop{\mathrm{Ker}}\nolimits}

\def\Hom{\mathop{\mathrm{Hom}}\nolimits}
\def\End{\mathop{\mathrm{End}}\nolimits}
\def\Ext{\mathop{\mathrm{Ext}}\nolimits}
\def\Tor{\mathop{\mathrm{Tor}}\nolimits}

\def\mod{\mathop{\mathrm{mod}}\nolimits}

\def\add{\mathop{\mathrm{add}}\nolimits}

\def\dd{\underline{\dim}}
\def\bx{x}
\begin{document}

\title{A Gabriel-type theorem for cluster tilting}
\author{Yuya Mizuno}
\thanks{The author is supported by Grant-in-Aid
for JSPS Fellowships No.23.5593.}
\address{Graduate School of Mathematics\\ Nagoya University\\ Frocho\\ Chikusaku\\ Nagoya\\ 464-8602\\ Japan}
\email{yuya.mizuno@math.nagoya-u.ac.jp}
\subjclass[2010]{16G20,\ 16E35}
\begin{abstract}
We study the relationship between $n$-cluster tilting modules over $n$ representation finite algebras and the Euler forms. 
We show that the dimension vectors of cluster-indecomposable modules give the roots of the Euler form.
Moreover, we show that cluster-indecomposable modules are uniquely determined by their dimension vectors. This is a generalization of Gabriel's theorem by cluster tilting theory. 
We call the above roots cluster-roots and investigate their properties. 
Furthermore, we provide the description of quivers with relations of $n$-APR tilts. Using this, 
we provide a generalization of BGP reflection functors. 

\end{abstract}
\maketitle
\tableofcontents
\section{Introduction}
Global dimension is one of the important invariants of algebras and it allows us to measure the homological complexity of algebras.
Algebras with global dimension 0 are semisimple. The next fundamental class consists of algebras of global dimension at most 1, which are characterized as path algebras  up to Morita equivalence.
With regard to this class, Gabriel gave the following famous theorem in 1972.
\begin{Theorem}[Gabriel]\label{Gab thm}
Let $K$ be an algebraically closed field, $Q$ be a connected Dynkin quiver and $q_Q$ be the Euler form (Definition \ref{C matrix}) of the path algebra $KQ$. 
\begin{itemize}
\item[(1)]
For any indecomposable $KQ$-module $X$, 
the dimension vector ${\dd} X$ is a positive root of $q_Q$.
\item[(2)]
The map $X \mapsto {\dd} X$ gives a bijection between the isomorphism classes of indecomposable $KQ$-modules and the positive roots of $q_Q$.
\end{itemize}
\end{Theorem}

This important theorem was investigated and developed by many authors. One of the generalizations was given by \cite{DF,N,DR,Ka}, where path algebras of non-Dynkin quivers were studied, and another was given by \cite{B2,D}, where representation-directed algebras of global dimension 2 were studied. 
We also point out that the notion of reflection functors \cite{BGP} was inspired by the theorem and it led to tilting theory \cite{APR,B1,BB,HR1,M,H,Ric}, which is an important tool in the study of many areas of mathematics. Thus, Gabriel's theorem plays a significant role in the representation theory of algebras.

The aim of this paper is to generalize Gabriel's theorem from the viewpoint of higher dimensional Auslander-Reiten theory, which was introduced by Iyama \cite{I1,I2,I3} and it 
has been developed by several authors \cite{EH,HI1,HI2,HIO,HZ,IO1,IO2}. It is also closely related to the recent studies of cluster tilting theory (see for example \cite{Ke1,Ke2,Re}). 
In higher dimensional Auslander-Reiten theory, \emph{$n$-representation-finite algebras} (Definition \ref{cluster-indec}) are fundamental objects. Recall that we call an algebra $n$-representation-finite if its global dimension is at most $n$ and it has an $n$-cluster tilting module, and we call a $\Lambda$-module \emph{cluster-indecomposable} if it is isomorphic to an indecomposable direct summand of an $n$-cluster tilting module. Note that 1-representation-finite algebras are exactly path algebras of Dynkin quivers and any indecomposable module is cluster-indecomposable in this case
(Remark \ref{1-rep-fin}). 

Let us investigate path algebras over a Dynkin quiver in the context of Auslander-Reiten theory.
For a Dynkin quiver $Q$, we have the following, where $\tau$ and $\tau^-$ are Auslander-Reiten translations.
\begin{itemize}
\item Any indecomposable $KQ$-module $X$ is isomorphic to $\tau^iI$ for an integer $i\geq0$ and an indecomposable injective $KQ$-module $I$.
\item Any indecomposable $KQ$-module $X$ is isomorphic to $\tau^{-i}P$ for an integer $i\geq0$ and an indecomposable projective $KQ$-module $P$.
\end{itemize}

Then, for an $n$-representation-finite algebra $\Lambda$, we have the following (Proposition \ref{cluster proper}), where $\tau_n$ and $\tau_n^-$ are $n$-Auslander-Reiten translations.
\begin{itemize}
\item Any cluster-indecomposable $\Lambda$-module $X$ is isomorphic to $\tau_n^iI$ for an integer $i\geq0$ and an indecomposable injective
 $\Lambda$-module $I$.
\item Any cluster-indecomposable $\Lambda$-module $X$ is isomorphic to $\tau_n^{-i}P$ for an integer $i\geq0$ and an indecomposable projective $\Lambda$-module $P$.
\end{itemize}
In this paper, we focus on this property and give a Gabriel-type theorem for algebras of   global dimension $n$.

Our main results are given by the following theorem.
\begin{Theorem}[Theorem \ref{result1}, \ref{uniqueness}]\label{main result}
Let $\Lambda$ be an $n$-representation-finite algebra and $q_\Lambda$ be the Euler form of $\Lambda$. 
\begin{itemize}
\item[(1)]
For any cluster-indecomposable $\Lambda$-module $X$, 
the dimension vector ${\dd} X$ is a positive root of $q_\Lambda$.
\item[(2)]
The map $X \mapsto {\dd} X$ gives an injection between the isomorphism classes of cluster-indecomposable $\La$-modules and the positive roots of $q_\La$.
\end{itemize}
\end{Theorem}

We call the above positive roots \emph{cluster-roots} (Definition \ref{cluster-root}).
Thus, cluster-roots correspond bijectively to cluster-indecomposable modules, 
which is an analogue of Theorem \ref{Gab thm}. 
In fact, when $n=1$, cluster-roots and cluster-indecomposable modules coincide with positive roots and indecomposable modules, respectively.
On the other hand, the above map is not necessarily surjective if $n>1$.
Therefore, our next aim is to characterize cluster-roots. 
We give a complete description of cluster-roots by \emph{reflection transformations} (Definition \ref{reflection}), which are linear transformations of the Grothendieck group.

\begin{Theorem}[Theorem \ref{form}]
Let $\Lambda$ be an $n$-representation-finite algebra and $Q$ be the quiver of $\Lambda$. Assume that $Q$ has an admissible ordering $(1,\ldots,l)$ (Definition \ref{admissible}). We denote by $t_i$ the reflection transformation associated with vertex $i$ and ${\bf c}:=t_l\ldots t_1$.
\begin{itemize}
\item[(1)]Any cluster-root can be written as
${\bf c}^{-u}{\bf p}_i$ for integers $u \geq 0$ and $i\in Q_0$, where ${\bf p}_i := t_1^{-1}\ldots t_{i-1}^{-1}{\bf e}_i$.
\item[(2)]Any cluster-root can be written as
${\bf c}^{v}{\bf q}_j$ for integers $v \geq 0$ and $j\in Q_0$, where ${\bf q}_j := t_l\ldots t_{j+1}{\bf e}_j$.
\end{itemize}
\end{Theorem}

Thus cluster-roots can be calculated by a method of linear algebra. 
Furthermore, we give the following terminology to characterize cluster-roots.

\begin{defi}[Definition \ref{C-positive def}]
We call $x \in \mathbb{Z}^l$ \emph{$\Phi$-sign-coherent} if $\Phi^{m}(x)\in \mathbb{Z}^l_{\geq0}$ or $\Phi^{m}(x)\in \mathbb{Z}^l_{\leq0}$ for any $m \in \mathbb{Z}$ and 
call $x \in \mathbb{Z}^l$ \emph{$\Phi$-positive} if $\Phi^{m}(x)\in \mathbb{Z}^l_{\geq0}$ for any $m \in \mathbb{Z}$, where $\Phi$ is the Coxeter transformation (Definition \ref{C matrix}).  
\end{defi}

Then we have the following result.

\begin{Theorem}[Theorem \ref{C-positive}]
Cluster-roots are $\Phi$-sign-coherent. 
Moreover, if $n$ is even, then cluster-roots are $\Phi$-positive.
\end{Theorem}

We consider the converse for $n=2$ and pose the following conjecture. 

\begin{conj}[Conjecture \ref{conjecture}]\label{conj0}
Let $\Lambda$ be a $2$-representation-finite algebra. 
Then $\Phi$-positive roots are cluster-roots.
This is equivalent, by Theorem \ref{main result}, to that $\Phi$-positive roots correspond bijectively to the isomorphism classes of cluster-indecomposable modules.
\end{conj}

We prove the conjecture for an important class of algebras.

\begin{theorem}[Corollary \ref{iterated}]
If $\Lambda$ is iterated tilted, then Conjecture \ref{conj0} is true. 
\end{theorem}
 
We now describe the organization of this paper.

In section \ref{preli}, we recall some definitions and results about $n$-representation-finite algebras and $n$-APR tilting modules. 
An $n$-APR tilting module over an $n$-representation-finite algebra induces an equivalence between certain nice subcategories (Corollary \ref{nearly}). This property plays a key role in this paper.

In section \ref{cluster-roots}, we explain the connection between the Coxeter transformation and $n$-AR translations. In particular, we show that the dimension vectors of cluster-indecomposable modules are roots of the Euler form.
In section \ref{Reflection transformations}, we introduce reflection transformations, which are linear transformations of the Grothendieck group. We will show that, in a certain subcategory, reflection transformations are compatible with the reflection functors. 

In section \ref{determine}, we define the Coxeter functors as a composition of reflection functors. Using them, we show that cluster-indecomposable modules can be constructed from a simple module by applying a sequence of reflection functors. This is a generalization of well-known properties of BGP reflection functors. As a corollary, we give a description of cluster-roots in terms of reflection transformations. Moreover we prove that cluster-indecomposable modules are uniquely determined by their dimension vectors.  

In section \ref{criterion}, we study a characterization of cluster-roots. We introduce the notions of $\Phi$-sign-coherence and $\Phi$-positivity. We show that cluster-roots are $\Phi$-sign-coherent roots. 
Moreover, we conjecture that $\Phi$-positive roots are cluster-roots for $n=2$. 
We show that this is true for iterated tilted algebras.

Section \ref{quivers} is independent from the previous sections. We give descriptions of $n$-APR tilts in terms of quivers with relations. Moreover, following \cite{BGP}, we give an interpretation of reflection functors in terms of linear representations of quivers. Then we explain the connection with the functors defined by $n$-APR tilting modules.\\

\textbf{Notations.}
Let $K$ be an algebraically closed field and we denote by $D:=\Hom_K(-,K)$. 
All algebras $\La$ are assumed to be basic, indecomposable and finite dimensional over $K$. 
All modules are right modules. We denote by $\mod\Lambda$ the category of finitely generated $\Lambda$-modules and by add$M$ the subcategory of $\mod\Lambda$ consisting of direct summands of finite direct sums of copies of $M$.
The composition $gf$ means first $f$, then $g$. We always assume that a quiver $Q$ is finite, connected and acyclic. 
We denote by $Q_0$ vertices of $Q$ and by $|Q_0|$ the number of $Q_0$. 
We denote by $Q_1$ arrows of $Q$. 
For an arrow or path $a$, we denote by $s(a)$ (respectively, $e(a)$) the start (respectively, the end) vertex of $a$. 
We denote by $KQ/(R)$ a path algebra $KQ$ bound by a set of relations $R$. 
We denote by $P_i, I_i, S_i$, respectively, the indecomposable projective, injective, simple $KQ/(R)$-module corresponding to the vertex $i \in Q_0$. 
\\

\textbf{Acknowledgement.}
First and foremost, the author would like to thank Osamu Iyama for his support and patient guidance.
Part of this work was done while he visited Trondheim. He would like to thank Idun Reiten and Steffen Oppermann for their suggestions and kind advice for correcting his English, and thank all the people at NTNU for their hospitality. 
He is grateful to Martin Herschend, Kota Yamaura and Takahide Adachi 
for their valuable advice. 


\section{Preliminaries}\label{preli}

In this section, we give a summary of some results we will use in this paper. 

Throughout this paper, we always assume that the quiver $Q$ is finite, connected and acyclic.

\subsection{Euler forms, Cartan matrix and Coxeter transformation}
We recall some terminology. We refer to \cite{ASS,Rin1} for details.

\begin{defi}\label{C matrix}
Let $\Lambda$ be a basic finite dimensional algebra with a complete set $\{e_1,\ldots,e_l\}$ of primitive orthogonal idempotents. Assume that $\Lambda$ has global dimension $n$.

$\bullet$ The \emph{Cartan matrix} of $\Lambda$ is the $(l,l)$-matrix
$$C_\Lambda =  (\dim(e_i\Lambda e_j))_{1\leq i,j\leq l}.$$

The \emph{Coxeter transformation} of $\Lambda$ is defined to be the matrix
$$\Phi_\La=\Phi := (-1)^nC_\Lambda^t C_\Lambda^{-1}.$$
Note that this definition is slightly different from the ordinary Coxeter transformation (cf. \cite[III.Definition 3.14]{ASS}). The difference is crucial in Lemma \ref{tau=c}.
 
$\bullet$ The \emph{Euler form} of $\Lambda$ is the $\mathbb{Z}$-bilinear form
$$\langle -, - \rangle : \mathbb{Z}^l \times \mathbb{Z}^l \to \mathbb{Z}$$
defined by $\langle x,y\rangle:= x^t(C_\Lambda^{-1})^t  y$
for $x ,y  \in \mathbb{Z}^l$.
We define a quadratic form $q_\Lambda(x) := \langle x,x\rangle$ and call a vector $x\in \mathbb{Z}^l$ \emph{root} if $q_\Lambda(x) = 1$.

$\bullet$ For a $\Lambda$-module $X$, we denote the \emph{dimension vector} by 
$${\dd} X :=(\dim (Xe_1),\ldots,\dim (Xe_l)).$$
\end{defi}

We recall the following properties.

\begin{lemm}\label{Euler form}
Let $\Lambda=KQ/(R)$ be a finite dimensional algebra of global dimension $n$.
\begin{itemize}
\item[(1)]We have $(C_\Lambda^tC_\Lambda^{-1})\dd P_i=\dd I_i$ for any $i\in Q_0$.
\item[(2)]We have $\langle x,y\rangle=(-1)^{n}\langle y,\Phi(x)\rangle=\langle \Phi(x),\Phi(y)\rangle.$
\item[(3)]For any $X,Y$ in $\mod\Lambda$, we have
$$ \langle {\dd}X, {\dd }Y\rangle  = \sum_{j \geq 0}(-1)^j\dim\Ext^j_\Lambda(X,Y).$$
\end{itemize}
\end{lemm}

\subsection{$n$-representation-finite algebras}

In this paper, $n$-representation-finite algebras are central objects.
They were investigated in \cite{HI1,I3,IO1,IO2} and, for $n=2$, in \cite{HI2}. 
See also \cite{I1,I2}.

\begin{defi}\label{cluster-indec}
We call a finite dimensional algebra $\Lambda$ \emph{$n$-representation-finite} (for a positive integer $n$)
if $\gl\Lambda\le n$ and there exists an \emph{$n$-cluster tilting $\Lambda$-module $M$}, i.e.
\begin{eqnarray*}
\add M&=&\{X\in\mod\Lambda\ |\ \Ext^i_\Lambda(M,X)=0\ \text{ for any }\ 0<i<n\},\\
&=&\{X\in\mod\Lambda\ |\ \Ext^i_\Lambda(X,M)=0\ \text{ for any }\ 0<i<n\}.
\end{eqnarray*}
We call a $\Lambda$-module $X$ \emph{cluster-indecomposable}
if $X$ is isomorphic to an indecomposable direct summand of $M$.
\end{defi}

\begin{remk}\label{1-rep-fin}
A basic 1-representation-finite algebra $\Lambda$ is a path algebra of a Dynkin quiver.
Indeed, an algebra $\Lambda$ is 1-representation-finite if and only if
it is representation-finite and hereditary.
Then, by Gabriel's Theorem, we have $\Lambda = KQ$ for some Dynkin quiver $Q$.
\end{remk}

Let $\Lambda$ be a finite dimensional algebra of global dimension $n$. 
We denote by $\DD^b(\mod\Lambda)$ the bounded derived category of $\mod\Lambda$.
Then, we define the Nakayama functor
$$\nu:=D\circ\RHom_\Lambda(-,\Lambda):\DD^b(\mod\Lambda)\to\DD^b(\mod\Lambda),$$
and we write 
$$\nu_n:=\nu\circ[-n]:\DD^b(\mod\Lambda)\to\DD^b(\mod\Lambda).$$
Note that $\Phi$ gives the action of $\nu_n$ on the Grothendieck group $K_0(\DD^b(\mod\Lambda))$.
Moreover, we define \emph{$n$-Auslander-Reiten translations} by the following functors
\begin{align*}
\tau_n &:= D\Ext_\La^n(-,\La):\mod\La\to\mod\La,\\
\tau_n^- & := \Ext_\La^n(D\La,-):\mod\La\to\mod\La.
\end{align*}
They are related to the functors $\nu_n$ and $\nu_n^{-1}$ as follows.
\begin{align*}
\tau_n &= H^0(\nu_n-):\mod\La\to\mod\La,\\
\tau_n^- &= H^0(\nu_n^{-1}-):\mod\La\to\mod\La.
\end{align*}

Then we recall some important properties of $n$-representation-finite algebras.
\begin{prop}\label{cluster proper}\cite{I1,I3}
Let $\Lambda=KQ/(R)$ be an $n$-representation-finite algebra and $M$ be a basic $n$-cluster tilting module. 
\begin{itemize}
\item[(1)]There exist non-negative integers $m_i$ and a bijection $\sigma:Q_0\to Q_0$ 
such that $\tau_n^{m_i}I_i\cong P_{\sigma(i)}$ for any $i\in Q_0$.
Moreover, we have $M\cong\bigoplus_{i\geq0}\tau_n^{-i}(\La)\cong\bigoplus_{i\geq0}\tau_n^{i}(D\La)$. 
\item[(2)]We have an equivalence $\tau_n:\add(M/\Lambda) \to\add(M/D\Lambda)$ with a quasi-inverse $\tau_n^-:\add(M/D\Lambda) \to\add(M/\Lambda).$
\item[(3)]There exist functorial isomorphisms for any $X, Y \in \add M$ and $Z\in \mod\La$.
$$ \underline{\Hom}_\Lambda(X,Z ) \cong D\Ext^n_\Lambda(Z, \tau_nX),\ \    \overline{\Hom}_\Lambda(Z,Y) \cong D\Ext^n_\Lambda(\tau_n^-Y, Z).$$
\item[(4)]We have $\Hom_\Lambda(\tau_n^i(D\Lambda),\tau_n^j(D\Lambda))=0$ for any $i<j$.
\item[(5)]$\tau_n$ and $\tau_n^-$ are nilpotent, i.e. $\tau_n^N=0$ and $\tau_n^{-N}=0$ for a sufficiently large integer $N$.
\item[(6)]We have $\tau_n\cong\nu_n$ on $\add(M/\Lambda)$ and $\tau_n^-\cong\nu_n^{-1}$ on $\add(M/D\Lambda)$.
\end{itemize}
\end{prop}


\subsection{$n$-APR tilting modules}\label{pre n-apr}

In \cite{IO1}, the authors introduced the notion of $n$-APR tilting modules as a generalization of APR tilting modules. In this subsection, we recall the definition and some properties of $n$-APR tilting modules.

\begin{defi}\label{def n-apr}\cite[Definition 3.1]{IO1} Let $\Lambda$ be a basic finite dimensional algebra, $k$ be a sink of the quiver of $\Lambda$ and $P_k$ be a simple projective $\Lambda$-module.
If ${\rm inj.dim} P_k = n$ and $\Ext^i_\Lambda(D\Lambda,P_k)=0$ for any $0 \leq i < n$, 
we call
$$ T_k := \tau_n^{-}P_k \oplus \Lambda/P_k$$
an \emph{$n$-APR tilting module}. Dually we define \emph{$n$-APR cotilting modules}.
\end{defi}

Clearly we have an ordinary APR tilting module if $n=1$.

Recall that we call a $\Lambda$-module $T$ \emph{tilting module of projective dimension $n$} \cite{M} if ${\rm proj.dim} T\leq n$, $\Ext_\Lambda^i(T,T)=0$ for $1\leq i\leq n$ and there exists an exact sequence
$0\rightarrow \Lambda \rightarrow T_0 \rightarrow\dots \rightarrow T_n \rightarrow0$ with $T_i$ in $\add T$.
Then we have the following important result.

\begin{prop}\cite[Theorem 3.2]{IO1}
An $n$-APR tilting module is a tilting module of projective dimension $n$.
\end{prop}
 
It is a classical result that $1$-representation-finite algebras (= path algebras of Dynkin quivers) are closed under taking endomorphism algebras of APR tilting modules. The following proposition is a generalization of this property and shows that $n$-APR tilting modules behave nicely over $n$-representation-finite algebras.
 
\begin{prop}\label{n-apr}\cite[Theorem 4.2, 4.7]{IO1}
Let $\Lambda$ be a basic $n$-representation-finite algebra and $k$ be a sink of the quiver of $\Lambda$. 
\begin{itemize}
\item[(1)] $P_k$ satisfies Definition \ref{def n-apr} and hence admits the $n$-APR tilting modules $T_k$.
\item[(2)] $\End_\Lambda(T_k)$ is an $n$-representation-finite algebra.
\item[(3)] $\Ext_\Lambda^n(T_k, {P_k})$ is isomorphic to the simple injective $\End_\Lambda(T_k)$-module corresponding to the vertex $k$.
\end{itemize}
\end{prop}

From this proposition, we know that a sink vertex $k$ becomes a source vertex of $\End_\Lambda(T_k)$.
A more detailed description of the quiver and relations of $\End_\Lambda(T_k)$ will be given in section \ref{quivers}.
By the above proposition, it is natural to ask for a relationship between $n$-cluster tilting modules of $\mod\Lambda$ and $\mod\End_\Lambda(T_k)$.
For this purpose, we recall a result of tilting theory.
 
Let $T$ be a tilting $\Lambda$-module and $\Gamma := \End_\Lambda(T)$. Let $\DD^b(\mod\Lambda)$ and $\DD^b(\mod\Gamma)$ be the bounded derived category of $\mod\Lambda$ and $\mod\Gamma$, respectively.
Then we have functors
\begin{align*}
\F &:=\RHom_\La(T,-):\DD^b(\mod\Lambda)\to\DD^b(\mod\Gamma),\\
\F_i &:=\Ext^i_\Lambda(T,-):\mod\La\to\mod\Gamma.
\end{align*}
we denote by 
\begin{eqnarray*}
\FF_i(T)=\FF_i:=\{X \in \mod \Lambda\ |\ \Ext^j_\Lambda(T,X) = 0 \ \text{for any} \ j \neq i\},\\
\XX_i(T)=\XX_i:=\{Y \in \mod \Gamma\ |\ \Tor_j^\Gamma(Y,T) = 0 \ \text{for any} \ j \neq i\}.
\end{eqnarray*}
First we recall the following result of Happel.

\begin{lemm}\label{tilt theorem}\cite{H}
In the above set-up, $\F :\DD^b(\mod\Lambda)\to\DD^b(\mod\Gamma)$ gives an equivalence and
the restriction of the functor $[i]\circ\F$ gives an equivalence $\F_i:\FF_i\to\XX_i.$
\end{lemm}

Next we give a relationship between the cluster tilting subcategories and $\FF_0$ as follows.

\begin{lemm}\label{CT and T}
Let $\Lambda =KQ/(R)$ be an $n$-representation-finite algebra, $k\in Q_0$ be a sink,
$T_k$ be the corresponding $n$-APR tilting $\Lambda$-module and $X$ be a cluster-indecomposable $\Lambda$-module. 
Then exactly one of the following holds.
\begin{itemize}
\item[(1)] $X \in \FF_0(T_k)$.
\item[(2)] $X\cong P_k$.
\end{itemize}
\end{lemm}

\begin{proof}
By Proposition \ref{cluster proper} (1), $T_k:=\tau_n^{-}P_k \oplus \Lambda/P_k$ is a direct summand of the cluster tilting module. 
Hence we have $\Ext_\La^i(T_k,X)=0$ for any $0< i < n$. Clearly we have $\Ext_\La^i(T_k,X)=0$ for any $i > n$ since $\gl\La\leq n$.
On the other hand, by Proposition \ref{cluster proper} (3), we have
$$\Ext^n_\Lambda (T_k , X) \cong \Ext^n_\Lambda (\tau_n^{-}P_k , X)
\cong D\overline{\Hom}_\Lambda (X ,P_k).$$
Assume that $X\in\FF_0(T_k)$. Since we have $\overline{\Hom}_\Lambda (X ,P_k)=0$, we obtain $X\ncong P_k$. 
Conversely if $X\ncong P_k$,  we obtain $\overline{\Hom}_\Lambda (X ,P_k)=0$ since $P_k$ is simple projective. Hence we have $X\in\FF_0(T_k)$.
\end{proof}

Then, by combining with Lemma \ref{tilt theorem} and \ref{CT and T}, we have the next corollary (cf. \cite[Theorem 4.2]{IO1}).

\begin{cor}\label{nearly}
Let $\Lambda =KQ/(R)$ be an $n$-representation-finite algebra, $k\in Q_0$ be a sink,
$T_k$ be the corresponding $n$-APR tilting $\Lambda$-module and $\Gamma := \End_\Lambda(T_k)$.
Denote by $M$ and $M'$ the basic $n$-cluster tilting modules of $\mod\Lambda$ and $\mod\Gamma$, respectively.
Then, the functors $\Hom_\Lambda(T_k,-)$ and $-\otimes_\Gamma T_k$ induce quasi-inverse equivalences 
$$\add (M/P_k) \cong \add (M'/I_k),$$
where $P_k$ and $I_k$ denote the simple projective $\Lambda$-module and the simple injective $\Gamma$-module associated with $k$, respectively.
\end{cor}

Furthermore, we need to recall the following definition.

\begin{defi}\label{admissible}
Let $Q$ be a finite, connected and acyclic quiver with $|Q_0|=l$. 
We call an ordering $i_1\ldots i_l$ of the vertices of $Q$ \emph{admissible} if for each $j$, the vertex $i_j$ is a sink for
the quiver obtained by removing vertices $i_1,\ldots, i_{j-1}$ 
from $Q$.
\end{defi}
Note that there exists an admissible ordering if $Q$ is acyclic. 
Throughout this paper, we assume that $Q$ has an admissible numbering $(1\ldots l)$ for simplicity.

Then, for an $n$-representation-finite algebra $\Lambda=KQ/(R)$, we have $n$-APR tilted algebras defined inductively by $\La^1:=\La,\La^2:=\End_{\La^1}(T_1),\ldots,\La^l:=\End_{\La^{l-1}}(T_{l-1})$ from Proposition \ref{n-apr}, where $T_{i}$ is the $n$-APR tilting $\Lambda^{i}$-module associated with $i\in Q_0$.  
Note that the quiver of $\La^i$ is also acyclic and has an admissible numbering. Explicit descriptions of $n$-APR tilted algebras are given in section \ref{quivers}.

\begin{exam}\label{exam0}
\begin{itemize}
\item[(1)]Let $\Lambda$ be the $n$-representation-finite algebra given by the following quiver with radical square zero relations.
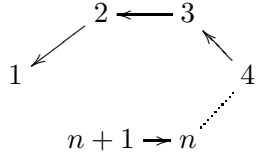
\[\xymatrix@C10pt@R10pt{   & 2 \ar[ld] & 3 \ar[l] & \\
1 &   &   & 4\ar[lu] &\\
 & n+1 \ar[r] &n \ar@{.}[ru]  &    }\]

Then we have $n$-APR tilted algebras as follows.

$
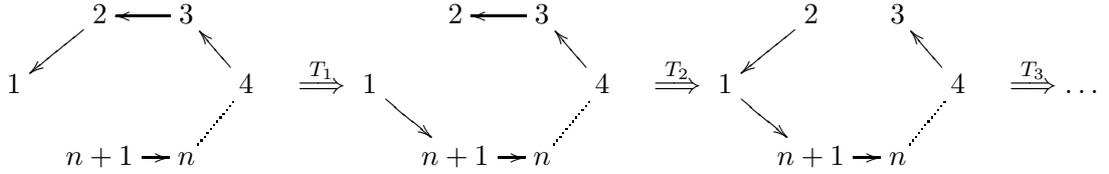
\xymatrix@C10pt@R10pt{   & 2 \ar[ld] & 3 \ar[l] &  \\
1 &   &   & 4\ar[lu] &\overset{T_1}{\Longrightarrow}\\
 & n+1 \ar[r] &n \ar@{.}[ru]  &   }
\xymatrix@C10pt@R10pt{   & 2  & 3 \ar[l] &  \\
1 \ar[rd]&   &   & 4\ar[lu] &\overset{T_2}{\Longrightarrow} \\
 & n+1 \ar[r] &n \ar@{.}[ru]  &  }
\xymatrix@C10pt@R10pt{   & 2 \ar[ld]  & 3 &  \\
1 \ar[rd]&   &   & 4\ar[lu] &\overset{T_3}{\Longrightarrow}\dots\\
 & n+1 \ar[r] &n \ar@{.}[ru]  &   }$
 
where all relations are radical square zero.
\item[(2)]
Let $\Lambda$ be the $2$-representation-finite algebra given by the following quiver with commutative relations and zero relations for each small half
square.
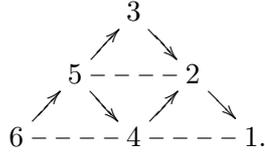
\[
\xymatrix@C10pt@R10pt{
 & & 3 \ar[rd]& & \\
 & 5\ar[ru] \ar@{--}[rr] \ar[rd]& & 2\ar[rd] &\\
6 \ar@{--}[rr]\ar[ru]& & 4 \ar@{--}[rr] \ar[ru] & & 1.}
\]
Then we have $2$-APR tilted algebras as follows.

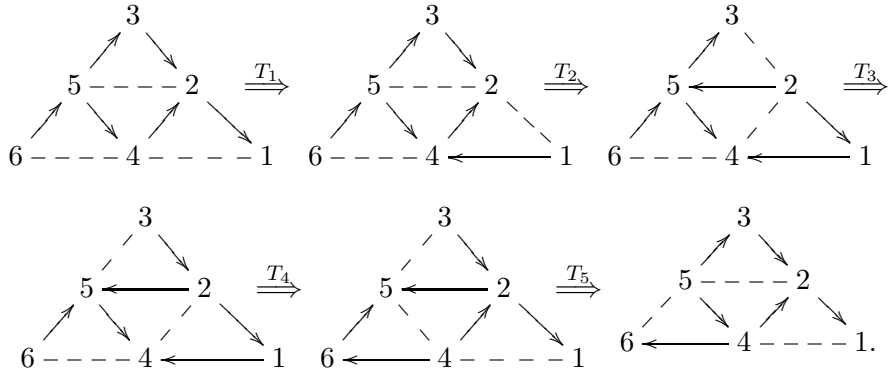
\[
\xymatrix@C10pt@R10pt{
 & & 3 \ar[rd]& & \\
 & 5\ar[ru] \ar@{--}[rr] \ar[rd]& & 2\ar[rd] &\overset{T_1}{\Longrightarrow} \\
6 \ar@{--}[rr]\ar[ru]& & 4 \ar@{--}[rr] \ar[ru] & & 1}
\xymatrix@C10pt@R10pt{
 & & 3 \ar[rd]& & \\
 & 5\ar[ru] \ar@{--}[rr] \ar[rd]& & 2\ar@{--}[rd] & \overset{T_2}{\Longrightarrow}\\
6 \ar@{--}[rr]\ar[ru]& & 4 \ar[ru] & & 1\ar[ll] }
\xymatrix@C10pt@R10pt{
 & & 3 \ar@{--}[rd]& & \\
 & 5\ar[ru]\ar[rd]& & 2\ar[rd] \ar[ll]  &\overset{T_3}{\Longrightarrow} \\
6 \ar@{--}[rr]\ar[ru]& & 4 \ar@{--}[ru] & & 1\ar[ll]}
\]
\[
\xymatrix@C10pt@R10pt{
 & & 3 \ar[rd]& & \\
 & 5\ar@{--}[ru]\ar[rd]& & 2\ar[rd] \ar[ll]  &\overset{T_4}{\Longrightarrow} \\
6 \ar@{--}[rr]\ar[ru]& & 4 \ar@{--}[ru] & & 1\ar[ll]}
\xymatrix@C10pt@R10pt{
 & & 3 \ar[rd]& & \\
 & 5\ar@{--}[ru]\ar@{--}[rd]& & 2\ar[rd] \ar[ll]  &\overset{T_5}{\Longrightarrow} \\
6 \ar[ru]& & 4 \ar[ru]\ar[ll] & & 1\ar@{--}[ll]}
\xymatrix@C10pt@R10pt{
 & & 3 \ar[rd]& & \\
 & 5\ar[ru]\ar[rd]& & 2\ar[rd] \ar@{--}[ll]  &\overset{}{} \\
6 \ar@{--}[ru]& & 4 \ar[ru]\ar[ll] & & 1.\ar@{--}[ll]}
\]
where all relations are commutative relations and zero relations for each small half
square.

\end{itemize}
\end{exam}

See \cite{IO1} for more examples. 


\section{Cluster-indecomposable modules and their properties}

\subsection{Cluster-roots}\label{cluster-roots}

The aim of this subsection is to show that the dimension vectors of cluster-indecomposable modules give roots of the Euler form. This property plays a crucial role in this paper.
Throughout this subsection, let $\Lambda=KQ/(R)$ be an $n$-representation-finite algebra.

First we give the following lemma, which is well-known for $n=1$ (see for example \cite[VIII. Proposition 2.2]{ARS}).
\begin{lemm}\label{tau=c}
Let $\Phi$ be the Coxeter transformation of $\La$ and
$X$ be a cluster-indecomposable $\Lambda$-module.
\begin{itemize}
\item[(1)]If $X$ is non-projective, then we have ${\dd}(\tau_n X) = \Phi({\dd}X).$
\item[(2)]If $X$ is non-injective, then we have ${\dd}(\tau_n^- X) = \Phi^{-1}({\dd}X).$
\end{itemize}
\end{lemm}

\begin{proof}
We only prove (1); the proof of (2) is similar.
Since $\Phi$ gives the action of $\nu_n$ on the Grothendieck group $K_0(\DD^b(\mod\Lambda))$,
 we have
${\dd}(\nu_n X)=\Phi({\dd}X)$.
On the other hand, we have $\nu_n X\cong\tau_n X$ by Proposition \ref{cluster proper} (6).
Thus we obtain
${\dd}(\tau_n X)={\dd}(\nu_n X)=\Phi({\dd}X).$
\end{proof}

By using Lemma \ref{tau=c}, we have the following result.

\begin{Theorem}\label{result1}
Let $\Lambda=KQ/(R)$ be an $n$-representation-finite algebra and
$q_\Lambda$ be the Euler form of $\Lambda$.
Then, for any cluster-indecomposable $\Lambda$-module $X$, the dimension vector ${\dd} X$ is a positive root of $q_\Lambda$.
\end{Theorem}

\begin{proof}
By Proposition \ref{cluster proper} (1),
there exist non-negative integers $j$ and $i \in Q_0$ such that
$$X \cong \tau_n^jI_i.$$
Then, by Lemma \ref{tau=c}, we have
$$ {\dd} X = \Phi^j({\dd} I_i).$$
Therefore, we have
\begin{eqnarray*}
q_\Lambda({\dd} X)&=&q_\Lambda(\Phi^j({\dd} I_i)) \\
&=&q_\Lambda({\dd} I_i) \ \ \ \ \ \ \ \ \ \ \ \ \ \ \ \ \ \ \ \ \ \ \ \  \ \ \ \ \ \ (\text{Lemma}\ \ref{Euler form}\ (2))\\
&=&\sum_{j \geq 0}(-1)^j\dim\Ext^j_\Lambda(I_i,I_i) \ \ \ \ \ \ \ \ \ \ (\text{Lemma}\ \ref{Euler form}\ (3))\\
&=&\dim\Hom_\Lambda(I_i,I_i) \\
&=&1\ \ \ \ \ \ \ \ \ \ \ \ \ \ \ \ \ \ \ \ \ \ \ \ \ \ \ \ \  \ \  \ \ \ \ \ \ \ \ \ \ \ \ \ \ (Q\ \text{is acyclic}).
\end{eqnarray*}
Thus ${\dd} X$ is a root of $q_\Lambda$.
\end{proof}

By this theorem, we use the following terminology.

\begin{defi}\label{cluster-root}
Let $\Lambda$ be an $n$-representation-finite algebra.
We call a root $x \in \mathbb{Z}^l$ \emph{cluster-root} if there is a cluster-indecomposable module $X$ such that ${\dd} X = x$.
\end{defi}

By Proposition \ref{cluster proper} (1) and Lemma \ref{tau=c}, we can easily calculate cluster-roots by applying the Coxeter transformation to the dimension vectors of indecomposable injective (or projective) modules. 

\begin{exam}\label{exam1}
Let $\Lambda$ be the following algebra with commutative relations and zero relations for each small half
square
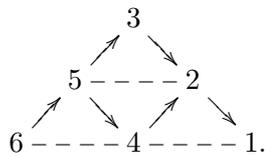
\[
\xymatrix@C10pt@R10pt{
 & & 3 \ar[rd]& & \\
 & 5\ar[ru] \ar@{--}[rr] \ar[rd]& & 2\ar[rd] & \\
6 \ar@{--}[rr]\ar[ru]& & 4 \ar@{--}[rr] \ar[ru] & & 1.
}
\]
Then the Coxeter transformation $\Phi$ of $\La$ is the matrix
\[
\Phi=C_\Lambda^t C_\Lambda^{-1}=\left( \begin{array}{cccccc}
1 & -1 &0 & 1 & 0 &0 \\
1 & 0 &-1 & 0 & 1 &0 \\
1 & 0 &0 &0 &0&0 \\
0 & 1 &-1 &0&0&1 \\
0 & 1 &0 &0&0&0 \\
0 & 0 &1 &0&0&0
\end{array} \right)_.\]
Then the cluster-roots are given as follows

$$\dd I_1=\left( \begin{smallmatrix}
1  \\
1 \\
1 \\
0 \\
0 \\
0
\end{smallmatrix}\right),
\dd I_2=\left( \begin{smallmatrix}
0  \\
1 \\
1 \\
1 \\
1 \\
0
\end{smallmatrix}\right),
\dd I_3=\left( \begin{smallmatrix}
0  \\
0 \\
1 \\
0 \\
1 \\
1
\end{smallmatrix}\right),
\dd I_4=\left( \begin{smallmatrix}
0  \\
0 \\
0 \\
1 \\
1 \\
0
\end{smallmatrix}\right),
\dd I_5=\left( \begin{smallmatrix}
0  \\
0 \\
0 \\
0 \\
1 \\
1
\end{smallmatrix}\right),\dd I_6=\left( \begin{smallmatrix}
0  \\
0 \\
0 \\
0 \\
0 \\
1
\end{smallmatrix}\right),$$ 

$$\dd P_2= \Phi(\dd I_4)=\left( \begin{smallmatrix}
1  \\
1 \\
0 \\
0 \\
0 \\
0
\end{smallmatrix}\right),
\dd P_4=\Phi(\dd I_5)=\left( \begin{smallmatrix}
0  \\
1 \\
0 \\
1 \\
0 \\
0
\end{smallmatrix}\right),$$

$$
\dd S_4=\Phi(\dd I_6)=\left( \begin{smallmatrix}
0  \\
0 \\
0 \\
1 \\
0 \\
0
\end{smallmatrix}\right),
\dd P_1=\Phi^2(\dd I_6)=\left( \begin{smallmatrix}
1  \\
0 \\
0 \\
0 \\
0 \\
0
\end{smallmatrix}\right).$$
\end{exam}

Thus we can easily obtain cluster-roots. Furthermore, as we will show in section \ref{determine}, cluster-indecomposable modules are uniquely determined by their dimension vectors.

For later use, we will show that the Coxeter transformation has finite order.

\begin{prop}\label{finite order}
Let $\Phi$ be the Coxeter transformation of $\Lambda$.
Then 
there exists an integer $d$ such that $\Phi^d={\bf 1}$.
\end{prop}

\begin{proof}
Since the global dimension of $\Lambda$ is finite,
the dimension vectors of the complete set $\{\dd P_1,\ldots,\dd P_l \}$ of the indecomposable projective modules generate the canonical basis of the free group $\mathbb{Z}^l$.
Therefore it is enough to show that there exists an integer $d$ such that $\Phi^d(\dd P_i) =\dd P_i$ for any $i\in Q_0$.

By \cite[Theorem 1.1]{HI1}, $\Lambda$ is a twisted fractionally Calabi-Yau algebra. 
Therefore there exist integers $s,t$ such that $\nu_n^s\Lambda\cong\Lambda[t]$ in $\DD^b(\mod\Lambda)$ (see \cite[Proposition 4.2]{HI1}). 
Thus, there exists an integer $d$ such that $\Phi^d(\dd P_i) =\dd P_i$ for any $i\in Q_0$.
\end{proof}


\subsection{Reflection transformations}\label{Reflection transformations}

In \cite{BGP}, the authors showed that reflection functors were compatible with the action of the Weyl group in the  Grothendieck group, 
and this property is crucial for their proof of Gabriel's theorem. 
We modify the classical reflection transformations to our setting, and show the compatibility with our reflection functors in a certain subcategory.
This property plays an essential role for establishing relationships between cluster-indecomposable modules and cluster-roots.

\begin{defi}\label{reflection}
Let $\Lambda=KQ/(R)$ be a finite dimensional algebra of global dimension $n$.
Then, for the Euler form $\langle -,-\rangle$ of $\Lambda$ (Definition \ref{C matrix}), the \emph{symmetrized Euler form} $(-,-)$
is defined by
$$ (x,y) :=\frac{1}{2}[\langle x,y\rangle  + \langle y,x\rangle ].$$
For each $i\in Q_0$,
we define $\mathbb{Z}$-linear map given by
$$s_i:\mathbb{Z}^l \to \mathbb{Z}^l,\ s_i(x) := x - 2(x,{\bf e}_i){\bf e}_i,$$
$$ \delta_i:\mathbb{Z}^l \to \mathbb{Z}^l,\ \delta_i(x):=(x_1,\ldots,x_{i-1},(-1)^{n-1} x_i,x_{i+1},\ldots,x_l).$$
Then, we define $t_i:=\delta_i s_i:\mathbb{Z}^l \to \mathbb{Z}^l$ and call this a \emph{reflection transformation}. 
\end{defi}

Next we define reflection functors via $n$-APR tilting modules as follows. 

Let $\Lambda=KQ/(R)$ be an $n$-representation-finite algebra, $k \in Q_0$ be a sink and $T_k$ be the  corresponding $n$-APR tilting module and $\Gamma:=\End_\La(T_k)$. 
We define \emph{reflection functors}
\begin{align*}
\T_k^+&:=\Hom_{\Lambda}(T_k,-):\mod\Lambda\to\mod\Gamma,\\
\T_k^-&:=-\otimes_{\Gamma} T_k:\mod\Gamma\to\mod\Lambda.
\end{align*} 

Note that we have $\T_k^+ X=0$ if and only if $X$ is isomorphic to the simple projective $\Lambda$-module $P_k$.
Moreover, if $X$ is a cluster-indecomposable $\Lambda$-module and $X\ncong P_k$, we have $\T_k^-\T_k^+ (X)\cong X$ by Corollary \ref{nearly}.

Then we give a relationship between the reflection functors and reflection transformations as follows. Here we use the notations $\FF_0$ and $\XX_0$ as defined in section \ref{pre n-apr}.

\begin{prop}\label{kakan1}
Let $\Lambda =KQ/(R)$ be an $n$-representation-finite algebra, $k \in Q_0$ be a sink and $T_k$ be the corresponding $n$-APR tilting module. Let $\Gamma:=\End_\La(T_k)$.
\begin{itemize}
\item[(1)]
Let $X$ be an indecomposable $\Lambda$-module. If $X \in \FF_0(T_k)$, then ${\T}_k^+X$ is an indecomposable $\Gamma$-module and we have ${\dd} ({\T}_k^+X) =t_k({\dd} X)$.
\item[(2)]
Let $Y$ be an indecomposable $\Gamma$-module. If $Y \in  \XX_0(T_k)$, then ${\T}_k^-Y$ is an
indecomposable $\Lambda$-module and we have ${\dd}({\T}_k^-Y) =t_k^{-1}({\dd} Y)$.
\end{itemize}
\end{prop}

\begin{proof}
We only prove (1); the proof of (2) is similar.
Since the functor ${\T}_k^+$ gives an equivalence between
$\FF_0(T_k)$ and $\XX_0(T_k)$ by Lemma \ref{tilt theorem}, ${\T}_k^+X$ is an indecomposable $\Gamma$-module.
We denote by $({\dd}X)_i$ the $i$-th coordinate of the dimension vector ${\dd}X$ for $i\in Q_0$.
We denote by $\epsilon_i$ the idempotent of $\Gamma$ corresponding to $i\in Q_0$.

First assume that $i \neq k$. 
By Lemma \ref{tilt theorem}, we have

\begin{eqnarray*}
({\dd}({\T}_k^+X))_i&=&\dim\Hom_\Gamma(\epsilon_i\Gamma,{\T}_k^+X)\\
&=&\dim\Hom_\Gamma(\Hom_\Lambda(T_k,e_i\Lambda),\Hom_\Lambda(T_k,X))\\
&=& \dim\Hom_\Lambda(e_i\Lambda,X)\\
&=& (\dim X)_i\\
&=&( t_k(\dim X))_i.
\end{eqnarray*}

Next assume that $i = k$. Similarly we have

\begin{eqnarray*}
({\T}_k^+X)e_k &\cong&\Hom_\Gamma( \epsilon_k\Gamma,{\T}_k^+X)\\
&\cong& \Hom_\Gamma(\Hom_\Lambda(T_k, \tau_n^{-}P_k), \Hom_\Lambda(T_k, X))\\
&\cong& \Hom_\Lambda(\tau_n^{-}P_k , X).
\end{eqnarray*}

On the other hand, we have
\begin{eqnarray*}
(s_k({\dd X}))_k&=&\dim Xe_k-[ \langle {\dd X},{\bf e}_k\rangle+ \langle {\bf e}_k,{\dd X}\rangle]\\
&=& - \langle {\dd X},{\bf e}_k\rangle.
\end{eqnarray*}

Then, we obtain
\begin{eqnarray*}
\langle{\dd X},{\bf e}_k\rangle&=&(-1)^n\langle \Phi^{-1}({\bf e}_k),{\dd X}\rangle\  \ \ \ \  \ \ \ \ \ \ \ \ \ \ \ \ \ \ \ \ \ \ (\text{Lemma}\ \ref{Euler form}\ (2))\\
&=& (-1)^n\langle \dd(\tau_n^-P_k),{\dd X}\rangle\ \ \ \ \ \ \ \ \ \ \ \ \ \ \ \ \ \ \ \ \ \ \ \ (\text{Lemma}\ \ref{tau=c})\\
&=& (-1)^n \sum_{j \geq 0}(-1)^j\dim\Ext^j_\Lambda(\tau_{n}^-P_k,X)\ \ \ \ \ \ \ (\text{Lemma}\ \ref{Euler form}\ (3)).
\end{eqnarray*}

Since $X\in\FF_0(T_k)$, we have
$\sum_{j \geq 0}(-1)^j\Ext^j_\Lambda(\tau_{n}^-P_k,X)=\Hom_\Lambda(\tau_{n}^-P_k,X).$
Hence we obtain 
\begin{eqnarray*}
(t_k({\dd X}))_k&=&(-1)^{n-1}(s_k({\dd X}))_k\\
&=& (-1)^{n-1}(-\langle{\dd X},{\bf e}_k\rangle)\\
&=& (-1)^{n-1}\{(-1)(-1)^n \dim\Hom_\Lambda(\tau_{n}^-P_k,X)\}\\
&=& \dim\Hom_\Lambda(\tau_{n}^-P_k,X).
\end{eqnarray*}
Thus we have shown that ${\dd} ({\T}_k^+ X) = t_k({\dd} X).$
\end{proof}
 
As an immediate consequence, we have the following. 
\begin{cor}\label{kakan2}
Let $\Lambda =KQ/(R)$ be an $n$-representation-finite algebra, $k \in Q_0$ be a sink and $T_k$ be the $n$-APR tilting module. Let $\Gamma:=\End_\Lambda(T_k)$.
\begin{itemize}
\item[(1)]
Let $X$ be a cluster-indecomposable $\Lambda$-module. If $X \ncong P_k$, then ${\T}_k^+X$
is an indecomposable $\Gamma$-module and we have ${\dd}({\T}_k^+X) =t_k({\dd} X).$
\item[(2)]
Let $Y$ be a cluster-indecomposable $\Gamma$-module. If $Y \ncong I_k$, then ${\T}_k^-Y$
is an indecomposable $\Lambda$-module and we have ${\dd}({\T}_k^-Y) =t_k^{-1}({\dd} Y).$
\end{itemize}
\end{cor}

\begin{proof}
We only prove (1); the proof (2) is similar. By Proposition \ref{kakan1}, it is enough to show $X\in\FF_0(T_k)$.
We have shown this in Lemma \ref{CT and T}.
\end{proof}

Next lemma is important in section \ref{criterion}. 
We note that, if $\gl\Lambda\leq 2$, then $q_\Lambda$ is given by the \emph{Tits form} \cite{B2} and it can be characterized by a quiver with relations.
 
\begin{lemm}\label{t_i}
Let $\Lambda =KQ/(R)$ be a $2$-representation-finite algebra with a minimal set $R$ of relations, $k \in Q_0$ be a sink and $T_k$ be the $2$-APR tilting module. Let $\Gamma:=\End_\Lambda(T_k)$.
\begin{itemize}
\item[(1)]We have $q_\Gamma(x)=q_\Lambda(\delta_kx)$.
\item[(2)]If x is a root of $q_\Lambda$ , then $t_kx$ is a root of $q_\Gamma$.
\end{itemize}
\end{lemm}

\begin{proof}
\begin{itemize}
Since $\gl\Lambda\leq 2$ and $\gl\Gamma\leq 2$, we can describe $q_\Lambda$ and $q_\Gamma$ by the Tits forms.
Because $k$ is a sink, we have
$$q_\Lambda(x)=\sum_{i\in Q_0}x_i^2 -  {\displaystyle\sum_{\begin{smallmatrix}a\in Q_1\\
e(a)=k\end{smallmatrix}}} x_{s(a)}x_{k} + {\displaystyle\sum_{\begin{smallmatrix}r\in R\\
e(r)=k\end{smallmatrix}}} x_{s(r)}x_k+\alpha,$$
where $\alpha$ consists of terms without having $x_k$. 
Let $(Q',R')$ be the quiver with a minimal set of relations satisfying  $KQ'/(R')\cong\Gamma$. 
Then by \cite[Theorem 3.11]{IO1} (cf. Proposition \ref{mutation}), we have
\begin{eqnarray*}
q_\Gamma(x)&=& \sum_{i\in Q_0}x_i^2 +  {\displaystyle\sum_{\begin{smallmatrix}r\in R'\\
s(r)=k\end{smallmatrix}}} x_{e(r)}x_{k} - {\displaystyle\sum_{\begin{smallmatrix}a\in Q_1'\\
s(a)=k\end{smallmatrix}}} x_{e(a)}x_k+\alpha\\
&=&\sum_{i\in Q_0}x_i^2 -  {\displaystyle\sum_{\begin{smallmatrix}a\in Q_1\\
e(a)=k\end{smallmatrix}}} x_{s(a)}(-x_{k}) + {\displaystyle\sum_{\begin{smallmatrix}r\in R\\
e(r)=k\end{smallmatrix}}} x_{s(r)}(-x_k)+\alpha.
\end{eqnarray*}
Thus the first statement follows.
Moreover it is well-known that $q_\Lambda(s_ix)=q_\Lambda(x)$ for any $i\in Q_0$ (\cite[VII.Lemma 3.7]{ASS}). 
Then, by the first assertion, we obtain
$q_\Gamma(t_k\bx)=q_\Lambda(\delta_kt_kx)=q_\Lambda(s_kx)=q_\Lambda(x)=1.$
\end{itemize}
\end{proof}

\begin{exam}\label{exam t}
Let $\Lambda$ be the algebra of Example \ref{exam1} and $\Gamma:=\End_\La(T_1)$ (the quiver and relations are given in Example \ref{exam0} (2)).
We have the Euler form of $\La$
$$q_\Lambda(x) = \sum_{i\in Q_0}x_i^2 -x_1x_2 -x_2x_3 -x_2x_4- x_3x_5 - x_4x_5 -x_5x_6 + x_1x_4 +x_2x_5 + x_4x_6.$$ 
Then since $q_\Gamma(x)=q_\La(\delta_1x),$ we have the Euler form of $\Gamma$ 
$$q_\Gamma(x)=\sum_{i\in Q_0}x_i^2 +x_1x_2 -x_2x_3 -x_2x_4- x_3x_5 - x_4x_5 -x_5x_6 - x_1x_4 +x_2x_5 + x_4x_6.$$
For example, ${\bf e}_4$ is one of the roots of $q_\La(x)$. Then for  
the reflection transformation
$t_1=\left(\begin{smallmatrix}
1&-1&0&1&0&0\\
0&1&0&0&0&0\\
0&0&1&0&0&0\\
0&0&0&1&0&0\\
0&0&0&0&1&0\\
0&0&0&0&0&1\\
\end{smallmatrix}\right)$, 
 $t_1{\bf e}_4$ is a root of $q_\Gamma(x).$
\end{exam}

\section{Cluster-roots determine cluster-indecomposable modules}\label{determine}
In this section we investigate relationships between cluster-indecomposable modules and cluster-roots in more detail. In particular, we show that cluster-indecomposable modules can be obtained from simple modules by applying a sequence of reflection functors.
As a corollary, we give a complete description of cluster roots in terms of reflection transformations. Moreover we prove that cluster-indecomposable modules are uniquely determined by their dimension vectors.

Throughout this section, let $\Lambda= KQ/(R)$ be an $n$-representation-finite algebra with an admissible numbering $(1\ldots l)$. As we have seen in subsection \ref{pre n-apr}, we have $n$-APR tilted algebras $\La^1:=\La,\La^2:=\End_{\La^1}(T_1),\ldots,\La^{l+1}:=\End_{\La^{l}}(T_{l})$. 
For $i\in Q_0$, we denote by $t_i$ the reflection transformation associated with $i$ of $\Lambda^i$. 
For $i,j\in Q_0$, we denote by $P_j^{(i)}$ the projective $\Lambda^i$-module corresponding to the vertex $j$. 
We write 
\begin{align*}
\T_i^+&:=\Hom_{\Lambda^i}(T_i,-):\mod\Lambda^i\to\mod\Lambda^{i+1},\\
\T_i^-&:=-\otimes_{\Lambda^{i+1}} T_i:\mod\Lambda^{i+1}\to\mod\Lambda^i,
\end{align*} 
 
and we write
\begin{align*}
\RT_i^+&:=\RHom_{\Lambda^i}(T_i,-):\DD^b(\mod\Lambda^i)\to\DD^b(\mod\Lambda^{i+1}),\\
\LT_i^-&:=-\Lotimes_{\Lambda^{i+1}} T_i:\DD^b(\mod\Lambda^{i+1})\to\DD^b(\mod\Lambda^{i}).
\end{align*}

First we give the following easy observation. 

\begin{lemm}\label{nu comm}
We have isomorphisms
$$\RT_i^+(\nu_n^{-1}P_j^{(i)})\cong\begin{cases}
\nu_n^{-1}P_j^{(i+1)} &(j\neq i)\\
P_j^{(i+1)}&(j=i),
\end{cases}$$
$$\LT_i^-(\nu_n^{-1}P_j^{(i+1)})\cong\begin{cases}
\nu_n^{-1}P_j^{(i)} &(j\neq i)\\
\nu_n^{-1}(\tau_n^-P_j^{(i)})&(j=i). 
\end{cases}$$
\end{lemm}

\begin{proof}
We only prove the first statement; the proof of the second one is similar. 
By the definition, we have 
$P_j^{(i+1)}=\Hom_{\Lambda^{i}}(T_{i},P_j^{(i)})$ for $j\neq i$ and 
$P_{i}^{(i+1)}=\Hom_{\Lambda^{i}}(T_{i},\tau_n^{-}P_{i}^{(i)})$. 
Assume that $j\neq i$. Then we have 
\begin{eqnarray*}
\RT_i^+(\nu_n^{-1}P_j^{(i)})
&\cong &\nu_n^{-1}(\RT_i^+(P_j^{(i)}))\\
&\cong & \nu_n^{-1}(\T_i^+(P_j^{(i)}))\ \ \ (\Ext_{\La^i}^m(T_i,P_j^{(i)})=0, m>0)\\
&\cong & \nu_n^{-1}(P_j^{(i+1)}).
\end{eqnarray*}

Assume that $j=i$. 
Then since $P_i^{(i)}$ is not an injective $\Lambda^i$-module, we have $\nu_n^{-1}P_j^{(i)}\cong\tau_n^-P_j^{(i)}$ by Proposition \ref{cluster proper} (6). 
Thus we obtain
\begin{eqnarray*}
\RT_i^+(\nu_n^{-1}P_j^{(i)})
&\cong &\RT_i^+(\tau_n^{-}P_j^{(i)})\\
&\cong &\T_i^+(\tau_n^{-}P_j^{(i)})\ \ \ (\Ext_{\La^i}^m(T_i,\tau_n^{-}P_j^{(i)})=0, m>0)\\
&\cong & P_j^{(i+1)}.
\end{eqnarray*}
\end{proof}

Then we give the next fundamental property of $n$-representation-finite algebras.  

\begin{prop}\label{rotation}
We have an isomorphism $\Lambda^{l+1}\cong\Lambda$.
\end{prop}

\begin{proof}
By applying Lemma \ref{nu comm} repeatedly, we can show that  
$\Lambda^{l+1}\cong{\RT}_{l}^+\ldots{\RT}_{1}^+(\nu_n^{-1}\Lambda).$ 
Since $\RT_i^+$ and $\nu_n^{-1}$ are equivalences,
we obtain isomorphisms
\begin{eqnarray*}
\Lambda^{l+1} &\cong&\End_{\Lambda^{l+1}}(\Lambda^{l+1})\\
&\cong& \End_{\DD^b(\mod\Lambda)}( {\RT}_{l}^+\ldots{\RT}_{1}^+(\nu_n^{-1}\Lambda))\\
&\cong&   \End_{\DD^b(\mod\Lambda)}( \nu_n^{-1}\Lambda) \\
&\cong& \End_{\DD^b(\mod\Lambda)}( \Lambda) \\
&\cong& \Lambda. 
\end{eqnarray*}
\end{proof}

The next lemma is analogous to \cite[Proposition 2.6]{APR}, which gives a relationship between the indecomposable injective modules over $\La^i$ and $\La^{i+1}$.

\begin{lemm}\label{inj form}
Fix a vertex $i\in Q_0$. 
Let $\{I_1,\ldots, I_l \}$ be a complete set of the indecomposable injective $\La^{i}$-modules. 
Then $\{\T_i^+(I_1),\ldots,\T_i^+(I_{i-1}),\Ext_\La^n(T_i,P_i),\T_i^+(I_{i+1}),\ldots,\T_i^+(I_l) \}$ is a complete set of the indecomposable injective $\La^{i+1}$-modules. 
\end{lemm}

\begin{proof}
Let $j\in Q_0$ be a vertex with $j\neq i$.
Then we obtain $\La^{i+1}$-module isomorphisms
\begin{eqnarray*}
\T_i^+(I_j)&\cong&\Hom_{\Lambda^{i}}(T_i,D(\La^i e_j))\\
&\cong& D(T_i\otimes_{\La^i}{\La^i} e_j)\\
&\cong& D((T_i)e_j).
\end{eqnarray*}
Moreover, we get 
\begin{eqnarray*}
(T_i)e_j&\cong&\Hom_{\Lambda^{i}}(e_j\La^{i},T_{i})\\
&\cong& \Hom_{\Lambda^{i+1}}(\Hom_{\Lambda^{i}}(T_i,e_j\La^{i}),\Hom_{\Lambda^{i}}(T_{i},T_{i}))\\
&\cong& \Hom_{\Lambda^{i+1}}(e_j\La^{i+1},\La^{i+1})\\
&\cong&\La^{i+1}e_j.
\end{eqnarray*}
Thus we obtain $\T_i^+(I_j)\cong D(\La^{i+1}e_j)$ as $\La^{i+1}$-modules.
Then, together with Proposition \ref{n-apr}, the statement follows.
\end{proof}

Then we give descriptions of all projective modules and injective modules in terms of reflection functors as follows. 

\begin{prop}\label{proj form}
\begin{itemize}
\item[(1)]Let $P_i$ be a projective $\Lambda$-module. Then $P_i$ is isomorphic to ${\T}_1^-\ldots{\T}_{i-1}^-(S_i)$, where $S_i$ is the simple projective $\Lambda^{i}$-module.
\item[(2)]Let $I_j$ be an injective $\Lambda$-module. Then $I_j$ is isomorphic to ${\T}_l^+\ldots{\T}_{j+1}^+(S_j)$, where $S_j$ is the simple injective $\Lambda^{j+1}$-module.
\end{itemize}
\end{prop}
 
\begin{proof}
\begin{itemize}
\item[(1)]
Let $h\in Q_0$ be a vertex such that $h<i$. 
Then since we have  $P_i^{(h)}=\Hom_{\Lambda^{h-1}}(T_{h-1},P_i^{(h-1)})=\T_{h-1}^+(P_i^{(h-1)})$,  
we have $\T_{h-1}^-P_i^{(h)}\cong \T_{h-1}^-(\T_{h-1}^+(P_i^{(h-1)}))\cong P_i^{(h-1)}$ by Corollary \ref{nearly}. Applying this process repeatedly for any $h<i$, we have the assertion. 
\item[(2)]
By Proposition \ref{rotation}, we have $\Lambda^{l+1}\cong\Lambda$. 
Then the proof is similar to that of (1) by Lemma \ref{inj form}.
\end{itemize}
\end{proof}

As a corollary, we will show that the Coxeter transformation can be described as a composition of reflection transformations.

\begin{cor}\label{c=-C}
Let $\Phi$ be the Coxeter transformation of $\La$.
Then we have $\Phi=t_l\ldots t_1$.
\end{cor}

\begin{proof}
It is enough to show that ${\Phi}({\dd}P_i) =t_l\ldots t_1({\dd}P_i)$ for any $i\in Q_0$.
By Lemma \ref{Euler form}, we have
$${\Phi}({\dd} P_i) = (-1)^n{\dd} I_i.$$
On the other hand, by Corollary \ref{kakan2} and Proposition \ref{proj form}, we have 
${\dd} P_i=t_1^{-1}\ldots t_{i-1}^{-1}{\bf e}_i$ and ${\dd} I_i=t_l\ldots t_{i+1} {\bf e}_i$.
Thus we get
\begin{eqnarray*}
t_l\ldots t_1({\dd} P_i)&=&t_l\ldots t_1 (t_1^{-1}\ldots t_{i-1}^{-1}{\bf e}_i)\\
&=&t_l\ldots t_i {\bf e}_i \\
&=& (-1)^{n}t_l\ldots t_{i+1} {\bf e}_i\\
&=&(-1)^{n} {\dd} I_i.
\end{eqnarray*}
\end{proof}

Next we give the connection between reflection functors and $n$-AR translations.

\begin{lemm}\label{composition}
We have an isomorphism 
$\tau_n^-\La\cong{\T}_{1}^-\ldots{\T}_{l-1}^-(T_l),$ where $T_l$ is the $n$-APR tilting $\La^l$-module.
\end{lemm}

\begin{proof}
We denote by $U_{i}:=( \bigoplus^{i-1}_{j = 1}P_j^{(i)})\oplus\tau_n^{-}( \bigoplus^l_{j = i} P_j^{(i)})$. 
Note that we have $U_1=\tau_n^-\La$ and $U_l=T_l$. 
Then it is enough to show that  
$$U_{i-1}\cong {\T}_{i-1}^-
U_{i},$$
for any $i\in Q_0$ with $i>1$. 
Fix a vertex $i\in Q_0$.

First, for a vertex $j\in Q_0$ such that $j<i-1$, we have $P_j^{(i-1)}\cong{\T}_{i-1}^-P_j^{(i)}$ by Corollary \ref{nearly}.  
Similarly, since ${\T}_{i-1}^+(\tau_n^{-}P_{i-1}^{(i-1)})\cong P_{i-1}^{(i)}$,
we have $\tau_n^{-}P_{i-1}^{(i-1)}\cong{\T}_{i-1}^-P_{i-1}^{(i)}$.

Next, let $j\in Q_0$ be a vertex such that $j\geq i$.
\begin{itemize}
\item[(i)]
Assume that $P_j^{(i)}$ is not injective. Then by Proposition \ref{cluster proper} (6), we obtain $\tau_n^-P_j^{(i)}\cong\nu_n^{-1}P_j^{(i)}$. Hence  we have 
\begin{eqnarray*}
{\T}_{i-1}^-(\tau_n^{-}P_j^{(i)}) 
&\cong & H^0(\LT_{i-1}^-(\nu_n^{-1}P_j^{(i)})) \\
&\cong & H^0(\nu_n^{-1}P_j^{(i-1)}) \ \ \ \ \ (\text{Lemma}\ \ref{nu comm})\\
&\cong & \tau_n^{-1}P_j^{(i-1)}.
\end{eqnarray*} 
\item[(ii)]
Assume that $P_j^{(i)}$ is injective. 
Then we can assume that $j>i$.
Hence we have $P_{j}^{(i)}\cong{\T}_{i-1}^+(P_{j}^{(i-1)})$.
It shows that $P_{j}^{(i-1)}$ is an injective $\La^{i-1}$-module by Lemma \ref{inj form}. 
Therefore we have $\tau_n^{-}P_j^{(i)}=0$ and $\tau_n^{-}P_j^{(i-1)}=0$.
\end{itemize}
Consequently, we get $\tau_n^{-}P_j^{(i-1)}\cong{\T}_{i-1}^-(\tau_n^{-}P_j^{(i)})$ for $j\geq i$. Thus we obtain $U_{i-1}\cong {\T}_{i-1}^-
U_{i}$ and the statement follows from the inductive step.
\end{proof}
 
Now we define the \emph{Coxeter functors} by
\begin{align*}
{\bf C}^+ &:= {\T}_l^+\ldots{\T}_1^+:\mod\Lambda\to \mod\Lambda,\\
{\bf C}^- &:= {\T}_1^-\ldots{\T}_l^-:\mod\Lambda\to \mod\Lambda.
\end{align*}

Then we explain the interplay between $n$-AR translations and the Coxeter functors. We remark that ${\bf C}^+,{\bf C}^-$ do not depend on the choice of an admissible numbering by the following result, which is an analogous result of \cite{APR}.

\begin{prop}\label{tau=coxeter2}

\begin{itemize}
\item[(1)]There exists an isomorphism ${\bf C}^+ X \cong \tau_n X$ for a $\Lambda$-module $X$.
\item[(2)]There exists an isomorphism  ${\bf C}^- X \cong \tau_n^- X$ for a $\Lambda$-module $X$.
\end{itemize}
\end{prop}
 
\begin{proof}
\begin{itemize}
\item[(1)]We have $\La$-module isomorphisms 
\begin{eqnarray*}
{\T}_l^+\ldots{\T}_1^+(X) &\cong&\Hom_{\Lambda^l}(T_l,\Hom_{\Lambda^{l-1}}(T_{l-1},\cdots(\Hom_{\Lambda^{}}(T_{1},X)\cdots)\\
&\cong& \Hom_{\Lambda}(T_l\otimes_{\Lambda^l}\ldots\otimes_{\Lambda^2} T_1,X)\\
&\cong& \Hom_{\Lambda}({\T}_1^-\ldots{\T}_{l-1}^-(T_l),X).
\end{eqnarray*}
Then, by Lemma \ref{composition}, 
we have ${\bf C}^+X  \cong \Hom_\Lambda (\tau_n^{-} \Lambda, X).$

Moreover we have $\Hom_\Lambda(\tau_n^{-}\Lambda,\Lambda)=0$ by Proposition \ref{cluster proper} (4). Thus 
we get $\Hom_\Lambda(\tau_n^{-}\Lambda,X)\cong\underline{\Hom}_\Lambda(\tau_n^{-}\Lambda,X)$. 
Then, 
we obtain
\begin{eqnarray*}
\Hom_\Lambda (\tau_n^{-} \Lambda, X) &\cong&\underline{\Hom}_\Lambda(\tau_n^{-}\Lambda,X)\\
&\cong& D\Ext^n_\Lambda(X,\tau_n(\tau_n^{-}\Lambda))\ \ \ \ \ \ \  \ (\text{Proposition}\ \ref{cluster proper}\ (3))\\
&\cong& D\Ext^n_\Lambda(X,\Lambda/I)\ \ \ \ \ (I:\text{injective modules of $\La$})\\
&\cong& D\Ext^n_\Lambda(X,\Lambda)\\
&=& \tau_n X.
\end{eqnarray*}

\item[(2)] We have isomorphisms
\begin{eqnarray*}
D({\T}_1^-\ldots{\T}_l^-(X) )&\cong&\Hom_{K}({\T}_1^-\ldots{\T}_l^-(X),K)\\
&\cong& \Hom_{K}(X\otimes_{\Lambda^{l+1}}T_l\otimes_{\Lambda^{l}}\ldots\otimes_{\Lambda^{2}} T_1,K)\\
&\cong& \Hom_{\Lambda}(X,\Hom_{K}(T_l\otimes_{\Lambda^{l}}\ldots\otimes_{\Lambda^{2}} T_1,K))\\
&\cong& \Hom_{\Lambda}(X,\Hom_{K}({\T}_1^-\ldots{\T}_{l-1}^-(T_l),K)).
\end{eqnarray*}

By Lemma \ref{composition} and $D(\tau_n^{-} \Lambda)=\tau_n^{}(D \Lambda)$, we have 
$D({\bf C}^-X)  \cong \Hom_\Lambda (X,D(\tau_n^{-} \Lambda))\cong\Hom_\Lambda (X,\tau_n^{}(D \Lambda)).$ 
Moreover we have $\Hom_\Lambda(D\Lambda,\tau_n(D\Lambda))=0$ by Proposition \ref{cluster proper} (4). Thus, we obtain
\begin{eqnarray*}
D\Hom_\Lambda (X,\tau_n^{} (D\Lambda)) &\cong&D\overline{\Hom}_\Lambda(X,\tau_n^{} (D\Lambda))\\
&\cong& \Ext^n_\Lambda(\tau_n^{-}(\tau_n^{}(D\Lambda)),X)\ \ \ \ \ \ \ \ \ \ \ \ \ (\text{Proposition}\ \ref{cluster proper}\  (3))\\
&\cong& \Ext^n_\Lambda((D\Lambda)/P,X)\ \ \ \ (P:\text{projective modules of $D\La$})\\
&\cong& \Ext^n_\Lambda(D\Lambda,X)\\
&=& \tau_n^{-} X.
\end{eqnarray*}
\end{itemize}
\end{proof} 
 
By Proposition \ref{tau=coxeter2}, if $X$ is projective (respectively, injective), then we obtain ${\bf C}^+X\cong\tau_nX=0$ (respectively, ${\bf C}^-X\cong\tau_n^{-}X=0$).
 
Using the above results, we show that any cluster-indecomposable module can be obtained from a simple module by applying a sequence of reflection functors. This is a generalization of \cite[Corollary 3.1]{BGP}.
 
\begin{cor}\label{description}
Let $X$ be a cluster-indecomposable $\Lambda$-module.
\begin{itemize}
\item[(1)]There exist integers $u \geq 0$ and $i \in Q_0$ such that
$X \cong {\bf C}^{-u}{\T}_1^-\ldots{\T}_{i-1}^-(S_i)$, where $S_i$ is the simple projective $\Lambda^i$-module.
\item[(2)]There exists integers $v \geq 0$ and $j \in Q_0$ such that
$X \cong {\bf C}^{v}{\T}_l^+\ldots{\T}_{j+1}^+(S_j)$, where $S_j$ is the simple injective $\Lambda^{j+1}$-module.
\end{itemize}
\end{cor}

\begin{proof}
We only prove (1); the proof of (2) is similar.
By Proposition \ref{cluster proper} (5) and Proposition \ref{tau=coxeter2}, there exists an integer $u \geq 0$ such that ${\bf C}^uX\neq0$ and ${\bf C}^{u+1}X=0$.
Hence there exists $i\in Q_0$ such that
${\T}_{i-1}^+{\T}_{i-2}^+\ldots{\T}_{1}^+{\bf C}^{u}(X) \neq 0$ and
${\T}_{i}^+{\T}_{i-1}^+\ldots{\T}_{1}^+{\bf C}^{u}(X) = 0$.

Since ${\T}_{i}^+:=\Hom_{\La^i}(T_i,-)$, we obtain ${\T}_{i-1}^+\ldots{\T}_{1}^+{\bf C}^{u}(X) \cong S_i$, where $S_i$ is the  simple projective $\Lambda^i$-module. 
By applying Corollary \ref{nearly}, we have $X \cong {\bf C}^{-u}{\T}_1^-\ldots{\T}_{i-1}^-(S_i)$.
\end{proof}

\begin{remk}
By Proposition \ref{proj form}, ${\T}_1^-\ldots{\T}_{i-1}^-(S_i)$ (respectively, ${\T}_l^+\ldots{\T}_{j+1}^+(S_j)$) is a projective module (respectively, injective module) for any $i\in Q_0$. 
On the other hand, the Coxeter functors ${\bf C}^+,{\bf C}^-$ give $n$-AR translations $\tau_n^{},\tau_n^{-}$ by Proposition \ref{tau=coxeter2}. 
These results imply that Corollary \ref{description} rephrases Proposition \ref{cluster proper} (1) using  reflection functors. 
\end{remk} 

As a corollary of the previous results, we give descriptions of cluster-roots in terms of reflection transformations as follows.

\begin{theorem}\label{form}Denote by ${\bf p}_i := t_1^{-1}\ldots t_{i-1}^{-1}{\bf e}_i, {\bf q}_j := t_l\ldots t_{j+1}{\bf e}_j$ and ${\bf c}:=t_l\ldots t_1$.
\begin{itemize}
\item[(1)]Any cluster-root can be written as
${\bf c}^{-u}{\bf p}_i$ for integers $u \geq 0$ and $i\in Q_0$.
\item[(2)]Any cluster-root can be written as
${\bf c}^{v}{\bf q}_j$ for integers $v \geq 0$ and $j\in Q_0$.
\end{itemize}
\end{theorem}
 
\begin{proof}
We only prove (1); the proof (2) is similar.
Let $x$ be a cluster-root.
By the definition, we can take a cluster-indecomposable $\La$-module $X$
such that ${\dd}X = x$.
By Corollary \ref{description}, there exist integers $u \geq 0$ and $i \in Q_0$ such that
$X \cong {\bf C}^{-u}{\T}_1^-\ldots{\T}_{i-1}^-(S_i)$, where $S_i$ is the simple projective $\Lambda^i$-module.
Then, applying Corollary \ref{kakan2} repeatedly, we obtain $\dd X = {\bf c}^{-u}t_1^{-1}\ldots t_{i-1}^{-1}(\dd S_i) = {\bf c}^{-u}{\bf p}_i.$
\end{proof}

\begin{exam}\label{exam2}
Let $\Lambda$ be the algebra of Example \ref{exam1} (the quiver with relations of $\La^i$ is given in Example \ref{exam0} (2)). 
Then we have 

$$t_1^{-1}=\left(\begin{smallmatrix}
1&1&0&-1&0&0\\
0&1&0&0&0&0\\
0&0&1&0&0&0\\
0&0&0&1&0&0\\
0&0&0&0&1&0\\
0&0&0&0&0&1\\
\end{smallmatrix}\right),
t_2^{-1}=\left(\begin{smallmatrix}
1&0&0&0&0&0\\
-1&1&1&1&-1&0\\
0&0&1&0&0&0\\
0&0&0&1&0&0\\
0&0&0&0&1&0\\
0&0&0&0&0&1\\
\end{smallmatrix}\right),
t_3^{-1}=\left(\begin{smallmatrix}
1&0&0&0&0&0\\
0&1&0&0&0&0\\
0&-1&1&0&1&0\\
0&0&0&1&0&0\\
0&0&0&0&1&0\\
0&0&0&0&0&1\\
\end{smallmatrix}\right),$$
$$t_4^{-1}=\left(\begin{smallmatrix}
1&0&0&0&0&0\\
0&1&0&0&0&0\\
0&0&1&0&0&0\\
1&-1&0&1&1&-1\\
0&0&0&0&1&0\\
0&0&0&0&0&1\\
\end{smallmatrix}\right),
t_5^{-1}=\left(\begin{smallmatrix}
1&0&0&0&0&0\\
0&1&0&0&0&0\\
0&0&1&0&0&0\\
0&0&0&1&0&0\\
0&1&-1&-1&1&1\\
0&0&0&0&0&1\\
\end{smallmatrix}\right),
t_6^{-1}=\left(\begin{smallmatrix}
1&0&0&0&0&0\\
0&1&0&0&0&0\\
0&0&1&0&0&0\\
0&0&0&1&0&0\\
0&0&0&0&1&0\\
0&0&0&1&-1&1\\
\end{smallmatrix}\right).$$

Thus we have
$${\bf c}^{-1}=t_1^{-1}\ldots t_l^{-1}=
\left( \begin{array}{cccccc}
0 & 0 &1 & 0 & 0 &0 \\
0 & 0 &0 & 0 & 1 &0 \\
0 & 0 &0 &0 &0&1 \\
1 & 0 &-1 &0&1&0 \\
0 & 1 &-1 &0&0&1 \\
0 & 0 &0 &1&-1&1
\end{array} \right)_.$$

Applying ${\bf c}^{-1}$ to ${\bf p}_i$, we have all cluster-roots as follows.

\[
{\bf p}_1=\left( \begin{array}{cccccc}
1  \\
0 \\
0 \\
0\\
0\\
0
\end{array} \right)\overset{{\bf c}^{-1}}{\to}\left( \begin{array}{cccccc}
0  \\
0 \\
0 \\
1\\
0\\
0
\end{array} \right)\overset{{\bf c}^{-1}}{\to}\left( \begin{array}{cccccc}
0  \\
0 \\
0 \\
0\\
0\\
1
\end{array} \right)\overset{{\bf c}^{-1}}{\to}\left( \begin{array}{cccccc}
0  \\
0 \\
1 \\
0\\
1\\
1
\end{array} \right)\overset{{\bf c}^{-1}}{\to}\left( \begin{array}{cccccc}
1  \\
1 \\
1 \\
0\\
0\\
0
\end{array} \right)\overset{{\bf c}^{-1}}{\to}{\bf p}_1,
\]
\[
{\bf p}_2=\left( \begin{array}{cccccc}
1  \\
1 \\
0 \\
0\\
0\\
0
\end{array} \right)\overset{{\bf c}^{-1}}{\to}\left( \begin{array}{cccccc}
0  \\
0 \\
0 \\
1\\
1\\
0
\end{array} \right)\overset{{\bf c}^{-1}}{\to}\left( \begin{array}{cccccc}
0  \\
1 \\
0 \\
1\\
0\\
0
\end{array} \right)\overset{{\bf c}^{-1}}{\to}\left( \begin{array}{cccccc}
0  \\
0 \\
0 \\
0\\
1\\
1
\end{array} \right)\overset{{\bf c}^{-1}}{\to}\left( \begin{array}{cccccc}
0  \\
1 \\
1 \\
1\\
1\\
0
\end{array} \right)\overset{{\bf c}^{-1}}{\to}{\bf p}_2.\]
Note that we have ${\bf p}_3=\left( \begin{array}{cccccc}
1  \\
1 \\
1 \\
0\\
0\\
0
\end{array} \right)_,\ {\bf p}_4=\left( \begin{array}{cccccc}
0  \\
1 \\
0 \\
1\\
0\\
0
\end{array} \right)_,\ {\bf p}_5=\left( \begin{array}{cccccc}
0  \\
1 \\
1 \\
1\\
1\\
0
\end{array} \right)_,\ {\bf p}_6=\left( \begin{array}{cccccc}
0  \\
0 \\
1 \\
0\\
1\\
1
\end{array} \right)_.$
\end{exam}

We can easily check that ${\bf c}^{-1}$ and ${\bf p}_i$ coincide with $\Phi^{-1}$ and ${\dd} P_i$, respectively. This follows from the results of Proposition \ref{proj form} and Corollary \ref{c=-C}.
 
Finally, we give the following theorem.
 
\begin{Theorem}\label{uniqueness}
Let $\Lambda= KQ/(R)$ be an $n$-representation-finite algebra.
Cluster-indecomposable modules are uniquely determined up to isomorphism by their dimension vectors.
\end{Theorem}

\begin{proof}
Assume that $X$ and $Y$ are cluster-indecomposable $\Lambda$-modules such that
${\dd} X = {\dd} Y$. By Corollary \ref{description},
there exist integers $u,v \geq 0$ and $i,j \in Q_0$ such that
${\T}_{i-1}^+\ldots{\T}_{1}^+{\bf C}^{u}(X) \cong S_i$ and 
${\T}_{j-1}^+\ldots{\T}_{1}^+{\bf C}^{v}(X) \cong S_j$,
where $S_i$ (respectively, $S_j$) is the simple projective $\Lambda^i$-module (respectively, simple projective $\Lambda^j$-module). 
We may suppose that $ul+i\leq vl+j$.
Using Corollary \ref{kakan2} repeatedly, 
we have
\begin{eqnarray*}
{\dd}({\T}_{i-1}^+\ldots{\T}_{1}^+{\bf C}^{u}(X))&=& t_{i-1}\ldots t_1{\bf c}^u({\dd} X)\\
&=&{\bf e}_i,
\end{eqnarray*}
where ${\bf c}=t_l\ldots t_1$. 
Similarly we have 
\begin{eqnarray*}
{\dd}({\T}_{i-1}^+\ldots{\T}_{1}^+{\bf C}^{u}(Y))&=& t_{i-1}\ldots t_1{\bf c}^u({\dd} Y)
\end{eqnarray*}
By ${\dd} X = {\dd} Y$, we have ${\dd}({\T}_{i-1}^+\ldots{\T}_{1}^+{\bf C}^{u}(Y))={\bf e}_i$. 
Hence we have ${\T}_{i-1}^+\ldots{\T}_{1}^+{\bf C}^{u}(Y)\cong S_i$, and $u=v$ and $i=j$. 
Therefore we have
$${\T}_{i-1}^+\ldots{\T}_{1}^+{\bf C}^{u}(X) \cong S_i \cong {\T}_{i-1}^+\ldots{\T}_{1}^+{\bf C}^{u}(Y).$$
By Corollary \ref{nearly}, we obtain
$X \cong Y$.
\end{proof}

\section{A criterion for cluster-roots}\label{criterion}

In previous sections, we investigate some properties of cluster-roots. 
As we can observe in Example \ref{exam4}, 
positive roots are not necessarily cluster-roots. 
Therefore it is natural to study a criterion for cluster-roots. 
In this section, we will show that cluster-roots have the $\Phi$-sign-coherent property and, if $n$ is even, they have $\Phi$-positive property. 
We conjecture that the converse is true for $n=2$ and show this for a certain class of algebras.

First, we introduce the following notions.
 
\begin{defi}\label{C-positive def}
Let $\Lambda$ be an $n$-representation-finite algebra and $\Phi$ be the Coxeter transformation of $\Lambda$.
We call $x \in \mathbb{Z}^l$ \emph{$\Phi$-sign-coherent} if $\Phi^{m}(x)\in \mathbb{Z}^l_{\geq0}$ or $\Phi^{m}(x)\in \mathbb{Z}^l_{\leq0}$ for any $m \in \mathbb{Z}$ and 
call $x \in \mathbb{Z}^l$ \emph{$\Phi$-positive} if $\Phi^{m}(x)\in \mathbb{Z}^l_{\geq0}$ for any $m \in \mathbb{Z}$.  
\end{defi}
By Proposition \ref{finite order}, the Coxeter transformation has finite order. Therefore 
$\Phi$-sign-coherent ($\Phi$-positive) property can be checked by finitely many calculations. 

Then we have the following consequence.

\begin{theorem}\label{C-positive}
Let $\Lambda$ be an $n$-representation-finite algebra. Then cluster-roots are $\Phi$-sign-coherent. Moreover, if $n$ is even, then cluster-roots are $\Phi$-positive. 
\end{theorem}
\begin{proof}
For any non-projective cluster-indecomposable $\La$-module $X$, 
we obtain $\Phi(\dd X) = \dd(\tau_n X)$ by Lemma \ref{tau=c}. 
Thus, we have $\Phi(\dd X)\in \mathbb{Z}^l_{\geq0}$.
On the other hand, for an indecomposable projective module $P_i$,
we have $\Phi (\dd P_i)=(-1)^n\dd I_i$ by Lemma \ref{Euler form}.
Thus the statement follows.
\end{proof} 

Now we pose the following conjecture.
 
\begin{conj}\label{conjecture}
Let $\Lambda$ be a $2$-representation-finite algebra.
Then $\Phi$-positive roots are cluster-roots. 
\end{conj}

We refer to \cite{HI2} for the structure of $2$-representation-finite algebras. 

\begin{exam}\label{exam4}
Let $\Lambda$ be the algebra of Example \ref{exam1}.
Then the Euler form $q_\La$ was given in Example \ref{exam t} and
all positive roots of $q_\Lambda$ are given as follows
\[
\left( \begin{smallmatrix}
x_1  \\
x_2 \\
x_3 \\
x_4 \\
x_5 \\
x_6
\end{smallmatrix} \right)=
\left( \begin{smallmatrix}
0  \\
0 \\
0 \\
0 \\
0 \\
1
\end{smallmatrix}\right)_,
\left( \begin{smallmatrix}
0  \\
0 \\
0 \\
0 \\
1 \\
1
\end{smallmatrix} \right)_,
\left( \begin{smallmatrix}
0  \\
0 \\
0 \\
1 \\
0 \\
0
\end{smallmatrix} \right)_,
\left( \begin{smallmatrix}
0  \\
0 \\
0 \\
1 \\
1 \\
0
\end{smallmatrix} \right)_,
\left( \begin{smallmatrix}
0  \\
0 \\
1 \\
0 \\
1 \\
1
\end{smallmatrix} \right)_,
\left( \begin{smallmatrix}
0  \\
1 \\
0 \\
1 \\
0 \\
0
\end{smallmatrix} \right)_,
\left( \begin{smallmatrix}
0  \\
1 \\
1 \\
1 \\
1 \\
0
\end{smallmatrix}\right)_,
\left( \begin{smallmatrix}
1  \\
0 \\
0 \\
0 \\
0 \\
0
\end{smallmatrix} \right)_,\]
\[\left( \begin{smallmatrix}
1  \\
1 \\
0 \\
0 \\
0 \\
0
\end{smallmatrix} \right)_,
\left( \begin{smallmatrix}
1  \\
1 \\
1 \\
0 \\
0 \\
0
\end{smallmatrix} \right)_,
\left( \begin{smallmatrix}
0  \\
1 \\
0 \\
0 \\
0 \\
0
\end{smallmatrix} \right)_,
\left( \begin{smallmatrix}
0  \\
0 \\
1 \\
0 \\
0 \\
0
\end{smallmatrix} \right)_,
\left( \begin{smallmatrix}
0  \\
0 \\
0 \\
0 \\
1 \\
0
\end{smallmatrix} \right)_,
\left( \begin{smallmatrix}
0  \\
1 \\
1 \\
0 \\
0 \\
0
\end{smallmatrix} \right)_,
\left( \begin{smallmatrix}
0  \\
1 \\
1 \\
1 \\
0 \\
0
\end{smallmatrix} \right)_,
\left( \begin{smallmatrix}
0  \\
0 \\
1 \\
1 \\
1 \\
0
\end{smallmatrix} \right)_,\left( \begin{smallmatrix}
0  \\
0 \\
1 \\
0 \\
1 \\
0
\end{smallmatrix} \right)_.\]

We can check that $\Phi$-positive roots are given as follows
\[
\left( \begin{smallmatrix}
0  \\
0 \\
0 \\
0 \\
0 \\
1
\end{smallmatrix}\right)_,
\left( \begin{smallmatrix}
0  \\
0 \\
0 \\
0 \\
1 \\
1
\end{smallmatrix} \right)_,
\left( \begin{smallmatrix}
0  \\
0 \\
0 \\
1 \\
0 \\
0
\end{smallmatrix} \right)_,
\left( \begin{smallmatrix}
0  \\
0 \\
0 \\
1 \\
1 \\
0
\end{smallmatrix} \right)_,
\left( \begin{smallmatrix}
0  \\
0 \\
1 \\
0 \\
1 \\
1
\end{smallmatrix} \right)_,
\left( \begin{smallmatrix}
0  \\
1 \\
0 \\
1 \\
0 \\
0
\end{smallmatrix} \right)_,
\left( \begin{smallmatrix}
0  \\
1 \\
1 \\
1 \\
1 \\
0
\end{smallmatrix}\right)_,
\left( \begin{smallmatrix}
1  \\
0 \\
0 \\
0 \\
0 \\
0
\end{smallmatrix} \right)_,\left( \begin{smallmatrix}
1  \\
1 \\
0 \\
0 \\
0 \\
0
\end{smallmatrix} \right)_,
\left( \begin{smallmatrix}
1  \\
1 \\
1 \\
0 \\
0 \\
0
\end{smallmatrix} \right)_.
\]
Then we can verify that they coincide with cluster-roots given in Example \ref{exam1}. 
Thus, the conjecture is true in this case.
\end{exam}

\begin{remk}
It is known that all roots of the Euler form of a Dynkin quiver are either 
non-negative roots or non-positive roots (for example \cite[VII.Lemma 4.8]{ASS}). 
Hence all roots are $\Phi$-sign-coherent in this case.
\end{remk}

\begin{remk}
For the case $n>2$, the $\Phi$-positivity is not enough to characterize cluster-roots.
For example, let $\La$ be the 4-representation-finite algebra given in Example \ref{exam0} (1). 
Then one can check that $\dd S_3={\bf e}_3$ is a $\Phi$-positive root of $q_\La$, while it is not a cluster-root.
\end{remk}

Throughout this section, unless otherwise specified, let $\Lambda= KQ/(R)$ be a $2$-representation-finite algebra with an admissible numbering $(1\ldots l)$ and with a minimal set $R$ of relations. We denote by $\La^1:=\La,\La^2:=\End_{\La^1}(T_1),\ldots,\La^{l+1}:=\End_{\La^{l}}(T_{l})$ the $2$-APR tilted algebras for $i\in Q_0$. 

To state the main result of this section, we introduce the following definition \cite{HR2}.

\begin{defi}\label{rep-direct}
Let $\Lambda$ be a finite dimensional algebra.
A \emph{path} in $\mod\Lambda$ is a sequence
$$X_0\xrightarrow{f_1}X_1\to\cdots\to X_{c-1}\xrightarrow{f_c}X_c$$
of nonzero non-isomorphisms $f_1,\ldots,f_c$ between indecomposable $\Lambda$-modules
$X_0,\ldots,X_c$. 
A path in $\mod\Lambda$ is called a \emph{cycle} if $X_0\cong X_c$. An algebra $\Lambda$ is called \emph{representation-directed} if there is no cycle in $\mod\La$.
\end{defi}

The aim of this section is to show the following theorem. 

\begin{theorem}\label{directed}
Let $\Lambda$ be a 2-representation-finite algebra.
If $\Lambda^i$ is representation-directed for any $i\in Q_0$, then Conjecture \ref{conjecture} is true.
\end{theorem}

As a corollary, we obtain the following result.

\begin{cor}\label{iterated}
Let $\Lambda$ be a $2$-representation-finite algebra.
If $\Lambda$ is iterated tilted, then Conjecture \ref{conjecture} is true.
\end{cor}

In this case, by Bongartz' theorem (Theorem \ref{B}), we have the following correspondence.
\begin{eqnarray*}
\{\text{Indecomposable modules}\}&\overset{1-1}{\longleftrightarrow }&\{\text{Positive roots}\}\\
\bigcup\ \ \ \ \ \ \ \ \ \ \ \ \ \ \ \ \ \ & &\ \ \ \ \ \ \ \ \ \bigcup\\
\{\text{Cluster-indecomposable modules}\}&\overset{1-1}{\longleftrightarrow }&  \{\text{$\Phi$-positive roots}\}.
\end{eqnarray*}

For a proof of Theorem \ref{directed}, we investigate some properties of $\Phi$-positive roots.

\begin{lemm}\label{c-positive}
Let $\Lambda$ be a 2-representation-finite algebra, 
$\Phi$ be the Coxeter transformation of $\La$ and $t_i$ be the reflection transformation associated with 
$i\in Q_0$ of $\Lambda^i$.
For $x\in\mathbb{Z}^l$,
the following conditions are equivalent.
\begin{itemize}
\item[(1)]$x$ is $\Phi$-positive.
\item[(2)]$t_i\ldots t_1\Phi^m(x)\in \mathbb{Z}^l_{\geq0}$ for any $i\in Q_0$ and $m\in\mathbb{Z}$.
\end{itemize}
\end{lemm}

\begin{proof}
By Corollary \ref{c=-C}, we have $\Phi=t_l\ldots t_1$.
Then clearly (2) implies (1).
Assume that $t_i\ldots t_1\Phi^m(x)\notin \mathbb{Z}^l_{\geq0}$ for some $i\in Q_0$ and $m\in\mathbb{Z}$.
Then there exists $p\in Q_0$ such that
$(t_i\ldots t_1\Phi^m(x))_p<0$, where $(t_i\ldots t_1\Phi^m(x))_p$ denotes by the $p$-th coordinate.

We assume that $p\leq i$.  
Since $t_j$ changes the only $j$-th coordinate for any $j\in Q_0$, we have
$$(\Phi^{m+1}(x))_p=(t_l\dots t_{i+1}(t_i\ldots t_1\Phi^m(x)))_p<0.$$
Similarly, if $p>i$, then we have $(\Phi^{m}(x))_p<0$.
Thus (1) implies (2).
\end{proof}

Next we give bijections between $\Phi_{\La^{i}}$-positive roots of $q_{\La^{i}}$ for any $i\in Q_0$.

\begin{prop}\label{number root}
We have the following bijections for any $i\in Q_0$
$$\{\Phi_{\La^i}\text{-positive roots of }q_{\La^i}\}\overset{1-1}{\longleftrightarrow }\{\Phi_{\La^{i+1}}\text{-positive roots of }q_{\La^{i+1}}\}$$ 
given by $$x\longmapsto t_ix.\ \ $$
\end{prop}
 
\begin{proof}
We only have to show for $i=1$. 
Let $x$ be a $\Phi_{\Lambda^{1}}$-positive root of $q_{\Lambda^{1}}$.
We will show that $t_1 x$ is a $\Phi_{\Lambda^{2}}$-positive root of $q_{\Lambda^{2}}$.
By Proposition \ref{rotation},
we have $\Phi_{\Lambda^{2}} = t_{1}t_l\ldots t_2$.
Since $x$ is a $\Phi_{\Lambda^{1}}$-positive, we have
$$\Phi_{\Lambda^1}^m (x) = (t_l\ldots t_1)^m (x) = t_l\ldots t_2\Phi_{\Lambda^{2}}^{m-1}(t_1 x) \in \mathbb{Z}^l_{\geq0}$$
for any $m\in \mathbb{Z}$.
By Lemma \ref{t_i}, $t_1 x$ is a root of $q_{\Lambda^{2}}$ 
and by Lemma \ref{c-positive}, it is $\Phi_{\Lambda^{2}}$-positive. 
Similarly, for a $\Phi_{\Lambda^{2}}$-positive root $y$ of $q_{\Lambda^{2}}$,
we can show that $t_1^{-1} y$ is a $\Phi_{\Lambda^{1}}$-positive root.
This implies the assertion.
\end{proof}
 
The above result is an analogue of bijections between cluster-indecomposable modules of $\mod\Lambda^i$, where the correspondence are given by reflection functors (Corollary \ref{nearly}). 
Thus, Proposition \ref{number root} implies that $\Phi$-positive roots behave like cluster-indecomposable modules.

As a corollary, we obtain the following result.

\begin{cor}
If Conjecture \ref{conjecture} is true for $\Lambda$, then it is also true for $\Lambda^i$ for any $i\in Q_0$.
\end{cor}

\begin{proof}
We only have to show for $i=2$.  
Let $y\in\mathbb{Z}^l$ be a $\Phi_{\Lambda^{2}}$-positive root of $q_{\Lambda^{2}}$. 
It is enough to show that there exists a cluster-indecomposable $\Lambda^2$-module $Y$ such that $\dd Y=y$.

By Proposition \ref{number root}, there exists $x\in\mathbb{Z}^l$ such that $y= t_1x$, where $x$ is a $\Phi_{\Lambda^{1}}$-positive root of $q_{\Lambda^{1}}$.
On the other hand, there exists a cluster-indecomposable $\Lambda$-module $X$ such that $\dd X= x$ by the assumption.
First, assume that $x\neq{\bf e}_1$, that is, $X\ncong P_1$.
Then, by Corollary \ref{nearly} and Corollary \ref{kakan2}, we have 
a cluster-indecomposable $\Lambda^2$-module $\T_1^+X$ such that $\dd(\T_1^+X)= t_1x$.
Next, assume that $x={\bf e}_1$ and hence $t_1x={\bf e}_1=y$.
Since the vertex 1 is a source of the quiver of $\Lambda^2$, we have the simple injective module $I_1$ and $\dd I_1=y$.
\end{proof}

Next, we prove the following key lemma.

\begin{lemm}\label{tilt cluster}
Let $\Lambda=KQ/(R)$ be an $n$-representation-finite algebra and $X$ be an indecomposable $\Lambda$-module.
For integers $u\geq0$ and $i\in Q_0$, if ${\T}_{i-1}^+\ldots{\T}_{1}^+{\bf C}^{u}(X)$ is either in $\FF_0(T_i)$ or isomorphic to the simple projective $\Lambda^i$-module $S_i$,
then $X$ is a cluster-indecomposable $\La$-module.
\end{lemm}

\begin{proof}
By Proposition \ref{cluster proper} (5) and Proposition \ref{tau=coxeter2}, 
we have ${\bf C}^{v+1}(X)=0$ for a sufficiently large integer $v\geq0$.
Thus there exist integers $v$ and $j\in Q_0$ such that
${\T}_{j-1}^+\ldots{\T}_{1}^+{\bf C}^{v}(X) \neq 0$ and ${\T}_{j}^+{\T}_{j-1}^+\ldots{\T}_{1}^+{\bf C}^{v}(X) = 0$. 
Hence we obtain ${\T}_{j-1}^+\ldots{\T}_{1}^+{\bf C}^{v}(X) \cong S_j$.

Since we have ${\T}_{j-2}^+\ldots{\T}_{1}^+{\bf C}^{v}(X)\ncong S_{j-1}$, it is in $\FF_0(T_{j-1})$ by the assumption.
Therefore we get ${\T}_{j-1}^+{\T}_{j-2}^+\ldots{\T}_{1}^+{\bf C}^{v}(X) \in\XX_0(T_{j-1})$ by Lemma \ref{tilt theorem}.
Hence we have
$${\T}_{j-1}^-(S_j)\cong{\T}_{j-1}^-({\T}_{j-1}^+{\T}_{j-2}^+\ldots{\T}_{1}^+{\bf C}^{v}(X)) \cong {\T}_{j-2}^+\ldots{\T}_{1}^+{\bf C}^{v}(X).$$
Applying this process repeatedly, we get 
$X \cong {\bf C}^{-v}{\T}_1^-\ldots{\T}_{j-1}^-(S_j)$.
Hence, by Corollary \ref{nearly}, $X$ is a cluster-indecomposable $\Lambda$-module.
\end{proof}

Next we recall an important result of representation-directed algebras \cite{B2}.

\begin{Theorem}[Bongartz]\label{B}
Let $\Lambda$ be a representation-directed algebra with $\gl\Lambda \leq 2$ and $q_\Lambda$ be
the Euler form of $\Lambda$.
The map $X \mapsto {\dd} X$ gives a bijection between the isomorphism classes of indecomposable $\Lambda$-modules
and the positive roots of $q_\Lambda$.
\end{Theorem}

Next lemma plays an important role to prove Theorem \ref{directed}.

\begin{lemm}\label{t gives} Let $\Lambda$ be a 2-representation-finite algebra and $x\in\mathbb{Z}^l$ be a $\Phi$-positive root of $q_{\Lambda}$. 
Assume that $\La$ is representation-directed. 
Then there exists an indecomposable $\Lambda$-module $X$ such that $\dd X=x$. Moreover, $X$ satisfies one of the following conditions. 
\begin{itemize}
\item[(1)] $X \in \FF_0(T_1)$.
\item[(2)] $X\cong P_1$.
\end{itemize}
\end{lemm}

\begin{proof}
By Theorem \ref{B}, there exists an indecomposable $\Lambda$-module $X$ such that $\dd X=x$.
We assume that $X \ncong P_1$ and show that $X \in \FF_0(T_1)$.

Then we have 
\begin{eqnarray*}
(t_1({\dd X}))_1&=&-\dim Xe_1+ \langle {\dd X},{\bf e}_1\rangle+ \langle {\bf e}_1,{\dd X}\rangle\\
&=&  \langle {\dd X},{\bf e}_1\rangle\\
&=&  (-1)^2\langle \Phi^{-1}({\bf e}_1),\dd X \rangle\ \ \ \ \ \ \ \ \ \ \ \ \ \ \ \    (\text{Lemma}\ \ref{Euler form}\ (2))\\
&=& \langle \dd(\tau_n^-P_1),{\dd X}\rangle\ \ \ \ \ \ \ \ \ \ \ \ \ \ \ \ \ \ \ \ \ \ \ \   (\text{Lemma}\ \ref{tau=c})\\
&=& \sum_{j \geq 0}(-1)^j\dim\Ext^j_\Lambda(\tau_{n}^-P_1,X)\ \ \ \ \ \ \ (\text{Lemma}\ \ref{Euler form}\ (3)).
\end{eqnarray*}

Therefore we obtain 
$$(t_1\bx)_1=\dim\Hom_\Lambda(\tau_2^-P_1,X)-\dim\Ext_\Lambda^1(\tau_2^-P_1,X)+\dim\Ext_\Lambda^2(\tau_2^-P_1,X).$$
By Proposition \ref{cluster proper} (3), we obtain
$\Ext_\Lambda^2(\tau_2^-P_1,X) \cong
D\overline{\Hom}_\Lambda(X,P_1).$ 
Therefore we have $\Ext_\Lambda^2(\tau_2^-P_1,X)=0$ since $P_1$ is a simple projective $\Lambda$-module and $X\ncong P_1$.

Now we assume that $X \notin \FF_0(T_{1})$. 

By $\Ext_\Lambda^2(\tau_2^-P_1,X)=0$, we have $\Ext_\Lambda^1(\tau_2^-P_1,X)\neq 0.$ 
On the other hand, we have $\Ext_\Lambda^1(\tau_2^-P_1,X)\cong D\overline{\Hom}_\Lambda(X,\tau(\tau_2^-P_1))$ by  Auslander-Reiten duality. 
Thus there exists a path from $X$ to $\tau(\tau_2^-P_1)$.
Moreover, $\tau_2^-P_1$ is an indecomposable module and it is not projective by Proposition \ref{cluster proper}. 
Therefore we have an almost split sequence
$$0\to\tau(\tau_2^-P_1)\to\ast\to\tau_2^-P_1\to0.$$
It implies that there exists a path from $\tau(\tau_2^-P_1)$ to $\tau_2^-P_1$.
Then since $\Lambda$ is representation-directed, there is no path from $\tau_2^-P_1$ to $X$ and we get $\Hom_\Lambda(\tau_2^-P_1,X)=0$.

Consequently we have $(t_1x)_1=\dim\Ext_\Lambda^1(\tau_2^-P_1,X)<0$, which contradicts the fact that $x$ is $\Phi$-positive from  Lemma \ref{c-positive}.
\end{proof}

Then we give a proof of Theorem \ref{directed}.
 
\begin{proof}
Let $x\in\mathbb{Z}^l$ be a $\Phi$-positive root of $q_{\Lambda}$.
By Theorem \ref{B}, there exists an indecomposable $\Lambda$-module $X$ such that $\dd X=x$.
We will show that $X$ is a cluster-indecomposable module. 
It is enough to show that $X$ satisfies the condition of Lemma \ref{tilt cluster}. 

If $X\cong P_1$, then there is nothing to show.  
If $X\ncong P_1$, then we have $X \in \FF_0(T_1)$ by Lemma \ref{t gives}. 
Then by Lemma \ref{tilt theorem} and Proposition \ref{kakan1}, $\T_1^+X$ is an indecomposable $\Lambda^2$-module and we have $\dd({\T}_1^+X)=t_1\bx$.

Moreover, by Proposition \ref{number root}, $t_1\bx$ is a $\Phi_{\Lambda^{2}}$-positive root of $q_{\Lambda^{2}}$ and we can apply Lemma \ref{t gives} repeatedly. 
Thus $X$ satisfies the condition of Lemma \ref{tilt cluster} and hence $X$ is a cluster-indecomposable $\Lambda$-module.
\end{proof}

Then Corollary \ref{iterated} follows immediately from this result.

\begin{prop}
Let $\Lambda$ be a $2$-representation-finite algebra. 
If $\Lambda$ is iterated tilted, then $\Lambda^i$ is representation-directed for any $i\in Q_0$.
\end{prop}
 
\begin{proof}
By \cite[Theorem 3.12]{IO2} (see also \cite{Rin2}), $\Lambda$ is derived equivalent to a path algebra of a Dynkin quiver. 
Therefore $\La$ and hence $\Lambda^i$ are representation-directed for any $i\in Q_0$.
\end{proof}

\begin{exam}\label{exam3}
\begin{itemize}
\item[(1)]
Let $\Lambda=KQ/(R)$ be the $2$-representation-finite algebra given by the following quiver and relations.
$$
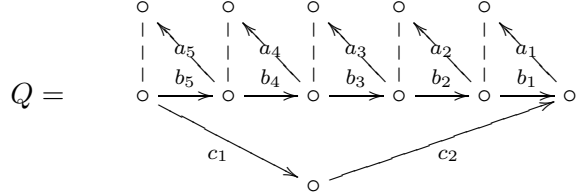
\xymatrix@C20pt@R20pt{ &\circ \ar@{--}[d]  & \circ \ar@{--}[d] &\circ\ar@{--}[d]&\circ\ar@{--}[d] &\ar@{--}[d] \circ  &  \\
Q=&\circ \ar[r]^{b_{5}}\ar[drr]_{c_1} & \circ \ar[r]^{b_{4}}\ar[ul]|{a_{5}}   &\ar[ul]|{a_{4}}\circ\ar[r]^{b_3}&\circ\ar[r]^{b_{2}}\ar[ul]|{a_3} &\ar[r]^{b_1}\ar[ul]|{a_2}   \circ  & \circ 
\ar[ul]|{a_1}\\  &&  &\circ\ar[urrr]_{c_2} & &  }$$
$$ R = \{c_1c_2-b_{5}b_4b_3b_2b_1, b_i a_i\}.$$
This algebra is iterated tilted by \cite[Theorem 3.12]{IO2}. 

Then there are 57 indecomposable $\La$-modules, and their dimension vectors give 
the following complete list of positive roots of $q_\La$.

$\left( \begin{smallmatrix}
00000\\
100000\\
0
\end{smallmatrix} \right),\left( \begin{smallmatrix}
00000\\
010000\\
0
\end{smallmatrix} \right),\left( \begin{smallmatrix}
00000\\
110000\\
0
\end{smallmatrix} \right),\left( \begin{smallmatrix}
00000\\
001000\\
0
\end{smallmatrix} \right),\left( \begin{smallmatrix}
00000\\
011000\\
0
\end{smallmatrix} \right),\left( \begin{smallmatrix}
00000\\
111000\\
0
\end{smallmatrix} \right),\left( \begin{smallmatrix}
00000\\
000100\\
0
\end{smallmatrix} \right),\left( \begin{smallmatrix}
00000\\
001100\\
0
\end{smallmatrix} \right),$
 
$\left( \begin{smallmatrix}
00000\\
011100\\
0
\end{smallmatrix} \right),\left( \begin{smallmatrix}
00000\\
111100\\
0
\end{smallmatrix} \right),\left( \begin{smallmatrix}
00000\\
000010\\
0
\end{smallmatrix} \right),\left( \begin{smallmatrix}
00000\\
000110\\
0
\end{smallmatrix} \right),\left( \begin{smallmatrix}
00000\\
001110\\
0
\end{smallmatrix} \right),\left( \begin{smallmatrix}
00000\\
011110\\
0
\end{smallmatrix} \right),\left( \begin{smallmatrix}
00000\\
111110\\
0
\end{smallmatrix} \right),\left( \begin{smallmatrix}
00000\\
000000\\
1
\end{smallmatrix} \right),$

$\left( \begin{smallmatrix}
00000\\
100000\\
1
\end{smallmatrix} \right),\left( \begin{smallmatrix}
00000\\
110000\\
1
\end{smallmatrix} \right),\left( \begin{smallmatrix}
00000\\
111000\\
1
\end{smallmatrix} \right),\left( \begin{smallmatrix}
00000\\
111100\\
1
\end{smallmatrix} \right),\left( \begin{smallmatrix}
00000\\
111110\\
1
\end{smallmatrix} \right),\left( \begin{smallmatrix}
00000\\
000001\\
0
\end{smallmatrix} \right),\left( \begin{smallmatrix}
00000\\
000011\\
0
\end{smallmatrix} \right),\left( \begin{smallmatrix}
00000\\
000111\\
0
\end{smallmatrix} \right),$

$\left( \begin{smallmatrix}
00000\\
001111\\
0
\end{smallmatrix} \right),\left( \begin{smallmatrix}
00000\\
011111\\
0
\end{smallmatrix} \right),\left( \begin{smallmatrix}
00000\\
000001\\
1
\end{smallmatrix} \right),\left( \begin{smallmatrix}
00000\\
000011\\
1
\end{smallmatrix} \right),\left( \begin{smallmatrix}
00000\\
000111\\
1
\end{smallmatrix} \right),\left( \begin{smallmatrix}
00000\\
001111\\
1
\end{smallmatrix} \right),\left( \begin{smallmatrix}
00000\\
011111\\
1
\end{smallmatrix} \right),\left( \begin{smallmatrix}
00000\\
111111\\
1
\end{smallmatrix} \right),$
 
$\left( \begin{smallmatrix}
10000\\
000000\\
0
\end{smallmatrix} \right),\left( \begin{smallmatrix}
10000\\
010000\\
0
\end{smallmatrix} \right),\left( \begin{smallmatrix}
10000\\
011000\\
0
\end{smallmatrix} \right),\left( \begin{smallmatrix}
10000\\
011100\\
0
\end{smallmatrix} \right),\left( \begin{smallmatrix}
10000\\
011110\\
0
\end{smallmatrix} \right),\left( \begin{smallmatrix}
10000\\
011111\\
0
\end{smallmatrix} \right),\left( \begin{smallmatrix}
10000\\
011111\\
1
\end{smallmatrix} \right),\left( \begin{smallmatrix}
01000\\
000000\\
0
\end{smallmatrix} \right),$

$\left( \begin{smallmatrix}
01000\\
001000\\
0
\end{smallmatrix} \right),\left( \begin{smallmatrix}
01000\\
001100\\
0
\end{smallmatrix} \right),\left( \begin{smallmatrix}
01000\\
001110\\
0
\end{smallmatrix} \right),\left( \begin{smallmatrix}
01000\\
001111\\
0
\end{smallmatrix} \right),\left( \begin{smallmatrix}
01000\\
001111\\
1
\end{smallmatrix} \right),\left( \begin{smallmatrix}
00100\\
000000\\
0
\end{smallmatrix} \right),\left( \begin{smallmatrix}
00100\\
000100\\
0
\end{smallmatrix} \right),\left( \begin{smallmatrix}
00100\\
000110\\
0
\end{smallmatrix} \right),$

$\left( \begin{smallmatrix}
00100\\
000111\\
0
\end{smallmatrix} \right),\left( \begin{smallmatrix}
00100\\
000111\\
1
\end{smallmatrix} \right),\left( \begin{smallmatrix}
00010\\
000000\\
0
\end{smallmatrix} \right),\left( \begin{smallmatrix}
00010\\
000010\\
0
\end{smallmatrix} \right),\left( \begin{smallmatrix}
00010\\
000011\\
0
\end{smallmatrix} \right),\left( \begin{smallmatrix}
00010\\
000011\\
1
\end{smallmatrix} \right),\left( \begin{smallmatrix}
00001\\
000000\\
0
\end{smallmatrix} \right),\left( \begin{smallmatrix}
00001\\
000001\\
0
\end{smallmatrix} \right),$

$\left( \begin{smallmatrix}
00001\\
000001\\
1
\end{smallmatrix} \right).$

Moreover there are 22 cluster-indecomposable $\La$-modules, and their dimension vectors give the following complete list of $\Phi$-positive roots. 

$\left( \begin{smallmatrix}
00001\\
000000\\
0
\end{smallmatrix} \right),\left( \begin{smallmatrix}
00010\\
000000\\
0
\end{smallmatrix} \right),\left( \begin{smallmatrix}
00100\\
000000\\
0
\end{smallmatrix} \right),\left( \begin{smallmatrix}
01000\\
000000\\
0
\end{smallmatrix} \right),\left( \begin{smallmatrix}
10000\\
000000\\
0
\end{smallmatrix} \right),\left( \begin{smallmatrix}
00001\\
000001\\
0
\end{smallmatrix} \right),\left( \begin{smallmatrix}
00001\\
000001\\
1
\end{smallmatrix} \right),\left( \begin{smallmatrix}
00010\\
000011\\
0
\end{smallmatrix} \right),$

$\left( \begin{smallmatrix}
00100\\
000111\\
0
\end{smallmatrix} \right),\left( \begin{smallmatrix}
01000\\
001111\\
0
\end{smallmatrix} \right),\left( \begin{smallmatrix}
10000\\
011111\\
0
\end{smallmatrix} \right),\left( \begin{smallmatrix}
00000\\
111111\\
1
\end{smallmatrix} \right),\left( \begin{smallmatrix}
00010\\
000010\\
0
\end{smallmatrix} \right),\left( \begin{smallmatrix}
00100\\
000100\\
0
\end{smallmatrix} \right),\left( \begin{smallmatrix}
01000\\
001000\\
0
\end{smallmatrix} \right),\left( \begin{smallmatrix}
10000\\
010000\\
0
\end{smallmatrix} \right),$

$\left( \begin{smallmatrix}
00000\\
100000\\
1
\end{smallmatrix} \right),\left( \begin{smallmatrix}
00000\\
111110\\
0
\end{smallmatrix} \right),\left( \begin{smallmatrix}
00000\\
111100\\
0
\end{smallmatrix} \right),\left( \begin{smallmatrix}
00000\\
111000\\
0
\end{smallmatrix} \right),\left( \begin{smallmatrix}
00000\\
110000\\
0
\end{smallmatrix} \right),\left( \begin{smallmatrix}
00000\\
100000\\
0
\end{smallmatrix} \right).$

\item[(2)]Let $\Lambda$ be the $2$-representation-finite algebra given as an Auslander algebra of $A_6$.
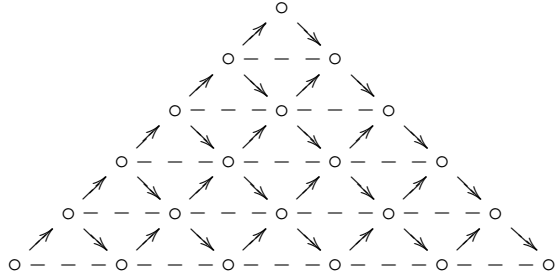
\[
\xymatrix@C8pt@R8pt{&& &&  & \circ\ar[rd]& &&&&\\
&& &&  \circ\ar@{--}[rr]\ar[ru]\ar[rd]& &\circ \ar[rd]&&&&\\
&& &\circ\ar@{--}[rr]\ar[ru]\ar[rd]&  &\circ\ar@{--}[rr]\ar[ru]\ar[rd] & &\circ\ar[rd]&&&\\
 &&\circ\ar@{--}[rr]\ar[ru]\ar[rd]& &\circ\ar@{--}[rr]\ar[ru]\ar[rd] &&\circ\ar@{--}[rr]\ar[ru]\ar[rd]&&\circ\ar[rd]&&\\
&\circ \ar@{--}[rr]\ar[ru]\ar[rd]& &\circ\ar@{--}[rr]\ar[ru]\ar[rd]&  &\circ \ar@{--}[rr]\ar[ru]\ar[rd]& &\circ\ar@{--}[rr]\ar[ru]\ar[rd]&&\circ\ar[rd]&\\
\circ \ar@{--}[rr]\ar[ru]& &\circ\ar@{--}[rr]\ar[ru]&  &\ar@{--}[rr]\ar[ru]\circ & &\circ\ar@{--}[rr]\ar[ru]&&\circ\ar@{--}[rr]\ar[ru]&&\circ.}
\]
Then there exist infinitely many indecomposable $\La$-modules and infinitely many roots of $q_\Lambda$.
Even in this case, using Theorem \ref{form}, one can check that there are 56 cluster-roots and 56 $\Phi$-positive roots (for example, by Mathematica). 
Thus cluster-roots correspond bijectively to $\Phi$-positive roots and the conjecture is ture.
\end{itemize}
\end{exam}

\section{Quivers of $n$-APR tilts and reflection functors}\label{quivers}

One of the important properties of APR tilting modules over a path algebra is that we can describe the quiver of its endomorphism algebra. 
More precisely, the new quiver is obtained by reversing the arrows adjacent to the sink 
of the original quiver \cite{BGP,APR}. 
In this section, we provide a generalization of this result to $n$-APR tilting modules and 
give an explicit description of $n$-APR tilts by quivers and relations. 
Moreover, using this, we introduce reflection functors in terms of linear representations of quivers. 

Throughout this section, let $Q$ be a finite, connected and acyclic quiver. Let 
$\Lambda =KQ/(R)$ be a finite dimensional algebra with a minimal set $R$ of relations. 
Assume that $k \in Q_0$ is a sink admitting the $n$-APR tilting module $T_k$.
Take a minimal injective resolution of $P_k$
\[\xymatrix{ 0 \ar[r] &  P_k \ar[r] & I^{0} \ar[r] & \ldots  \ar[r] & I^{n-1}  \ar[r] & I^n \ar[r] &0.   }\]
Then, by \cite[Theorem 3.2]{IO1}, we have a minimal projective resolution of $\tau_n^{-}P_k$ 
\[\tag{i}\label{exact seq}\xymatrix{ 0 \ar[r] & P^{0} \ar[r] & \ldots \ar[r]&P^{n-1} \ar[r]^{ \alpha} & P^{n} \ar[r]^{\beta} & \tau_n^{-}P_k \ar[r] &0,   }\]
where $P^i=\nu^{-1}(I^i)$ for $0\leq i\leq n$. Note that we have $P^{i}\in \add (\Lambda/P_k)$ for $1\leq i\leq n$ since $P_k$ is a simple projective $\Lambda$-module.

We decompose $P^{n-1}$ and $P^{n}$ into direct sums of indecomposable modules and denote them, respectively, by $\bigoplus_{b\in B}P_{i_b}$ and $\bigoplus_{c\in C}P_{j_c}$. 
We also denote the morphism $\alpha$ by $(a_{cb})_{cb}: \bigoplus_{b\in B}P_{i_b}\to \bigoplus_{c\in C}P_{j_c}$. 
We define new arrows $a_c^*:k\to {j_c}$ for each $c\in C$.
Then we give the following definition.
 
\begin{defi}\label{2-apr-tilt}
Under the above set-up, we define a new quiver with relations $(Q',R')=\sigma_k(Q,R)$ as follows.
\begin{itemize}
\item[(1)]$Q'_0 = Q_0$
\item[(2)]$Q_1'= \{a \in Q_1\ |\ e(a) \neq k\} \coprod \{\ a_c^*:k\to {j_c}\ |\ c\in C\}.$
\item[(3)]
$R'=\begin{cases}
0 &\text{if}\ n=1\\
\{r \in R\ |\ e(r) \neq k\}\coprod \{\ \sum_{c\in C}a_c^*a_{cb}\ |\ b\in B \}&\text{if}\ n>1.
\end{cases}$

\end{itemize}
Dually, for a source $k$ admitting the $n$-APR co-tilting module, 
we define $\sigma_k^-$.
\end{defi}

Note that if $n=1$, $Q'$ is obtained by just reversing the arrows. Moreover, 
if $n=2$, $(Q',R')$ is determined by a simple method (cf. \cite[section 3.3]{IO1}).
 
\begin{exam}\label{exam ref}
\begin{itemize}
\item[(1)]
Let $\Lambda=KQ/(R)$ be the algebra given by the following quiver with relations
$$
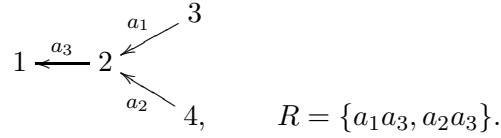
\xymatrix@C20pt@R5pt{& &3 \ar[dl]_{a_1} &   \\
1&2 \ar[l]_{a_3}&   &    \\
& &4, \ar[ul]^{a_2}& R=\{a_1a_3,a_2a_3\}. }$$ 
Then $P_1$ admits a 2-APR tilting module and 
we have the following projective resolution 
\[\xymatrix{ 0 \ar[r] & P_{1}  \ar[r]&P_{2}\ar[r]& P_{3}\oplus P_{4}\ar[r]^{ } & \tau_2^{-}P_1 \ar[r] &0.   }\]
Then $(Q',R')=\sigma_1(Q,R)$ is given as follows
$$
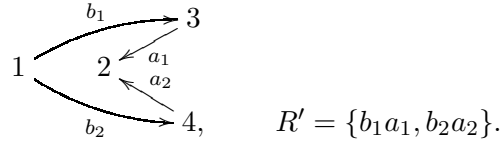
\xymatrix@C20pt@R5pt{& &3 \ar[dl]^{a_1} &   \\
1\ar@/^2mm/[urr]^{b_1} \ar@/_2mm/[drr]_{b_2}&2 &   &  &\\
& &4, \ar[ul]_{a_2}&  R'=\{b_1a_1,b_2a_2\}.}$$ 

\item[(2)]
Let $\Lambda=KQ/(R)$ be the algebra given by the following quiver with relations
$$
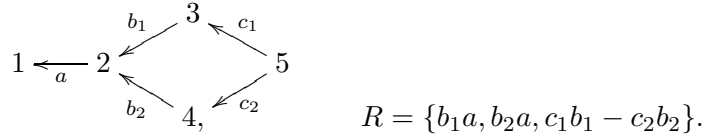
\xymatrix@C20pt@R5pt{& &3 \ar[dl]_{b_1} &  & \\
1&2\ar[l]^a &   &5\ar[ul]_{c_1}\ar[dl]^{c_2}&&\\
& &4, \ar[ul]^{b_2}& &R=\{b_1a,b_2a,c_1b_1-c_2b_2\}.}$$

Then $P_1$ admits a 3-APR tilting module and we have the following projective resolution 
\[\xymatrix{ 0 \ar[r] & P_{1}  \ar[r]&P_{2} \ar[r]^{ } & P_{3}\oplus P_{4} \ar[r]^{} & P_{5}  \ar[r] & \tau_3^{-}P_1 \ar[r] &0.   }\]
Then $(Q',R')=\sigma_1(Q,R)$ is given as follows 
$$
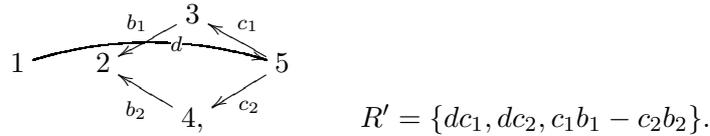
\xymatrix@C20pt@R5pt{& &3 \ar[dl]_{b_1} &  & \\
1\ar@/^3mm/[rrr]|(0.6)d&2 &   &5\ar[ul]_{c_1}\ar[dl]^{c_2}&&\\
& &4, \ar[ul]^{b_2}& & R'=\{dc_1,dc_2,c_1b_1-c_2b_2\}.}$$
\end{itemize}
\end{exam}
 
Then we give the following result.

\begin{prop}\label{mutation}
Let $\Lambda =KQ/(R)$ be a finite dimensional algebra with a minimal set $R$ of relations. Assume that $k \in Q_0$ is a sink admitting the $n$-APR tilting module $T_k$ and $(Q',R'):=\sigma_k(Q,R)$.
Then there is an isomorphism $\End_\Lambda(T_k) \cong KQ'/(R')$ and $R'$ is a minimal set of relations.
\end{prop}

We remark that $KQ'/(R')$ is $n$-representation-finite if $\La$ is by Proposition \ref{n-apr}. It allows us to give explicit descriptions of $n$-representation-finite algebras by $n$-APR tilts (Example \ref{exam0}). 

For a proof, we recall the following result. 
We denote by $J_\Lambda$ the Jacobson radical of an algebra $\Lambda$ and call an element of $KQ$ \emph{basic} if it is a linear sum of paths in $Q$ with a common start and a common end. We refer to \cite{BIRS} for more details. 
 
\begin{theorem}\label{birs}\cite[Proposition 3.1, 3.3]{BIRS}
Let $Q$ be a finite, connected and acyclic quiver and $\Gamma$ be a basic finite dimensional algebra.
Let $\phi: {KQ} \to \Gamma$ be an algebra homomorphism and
$R$ be a finite set of basic elements in $J_{{KQ}}$.
Then the following conditions are equivalent.
\begin{itemize}
\item[(1)]$\phi$ is surjective and $\Ker\phi={I}$ for the ideal $I=(R)$ of ${{KQ}}$.
\item[(2)]
The following sequence is exact for any $i\in Q_0$
\[\xymatrix@C40pt{ \displaystyle\bigoplus_{\begin{smallmatrix}r \in R,s(r)=i\end{smallmatrix}}(\phi e(r))\Gamma \ar[r]^{(\phi\partial_ar)_{r,a}}  &
{\displaystyle\bigoplus_{\begin{smallmatrix}a \in Q_1,s(a)=i\end{smallmatrix}}} (\phi e(a))\Gamma \ar[r]^{\ \ \ \ \ \  \ (\phi a)_a} &  (\displaystyle{\phi i})J_\Gamma \ar[r]& 0.  }\]
\end{itemize}
\end{theorem}

Then we give a proof of Proposition \ref{mutation}.

\begin{proof}
Take the minimal projective resolution of $\tau_n^-P_k$ (\ref{exact seq}) and we follow the same notations. 
Let $\Gamma:=\End_\Lambda(T_k)$. 
We denote by $\phi$ the natural surjective map $KQ\to\Lambda$, which gives $\Ker\phi=(R)$.
Then we define an algebra homomorphism  $\phi':{KQ'} \to \Gamma$ as follows.   
\begin{itemize}
\item[(a)]For $i\in Q_0'$, we define $\phi'i=\iota_i\pi_i\in\Gamma$, where $\pi_i:T_k\to P_i$ (respectively, $\iota_i:P_i\to T_k$) denotes the canonical projection (injection) for $i\neq k$ and 
$\pi_i:T_k\to \tau_n^-P_k$ (respectively, $\iota_i:\tau_n^-P_k\to T_k$) denotes the canonical projection (injection) for $i=k$.

\item[(b)]For $a\in Q_1'\cap Q_1$, define $\phi' a =\phi a$. For $a^*_c$ with $c\in C$, define 
$$\phi'((a_c^*)_{c\in C})= \beta \in \Hom_\Lambda(\bigoplus_{c\in C}P_{j_c},\tau_n^-P_k),$$
where $\beta$ is given in (\ref{exact seq}).
\end{itemize}

Take a vertex $i\in Q_0$ with $i\neq k$ and 
let $S_i$ be the simple $\Lambda$-module. Then we have a minimal projective presentation of $S_i$ as follows
\[\tag{ii}\label{presen 0}\xymatrix{  {\displaystyle\bigoplus_{\begin{smallmatrix}r \in R,s(r)=i\end{smallmatrix}}} P_{e(r)} \ar[r]^{f_i} &{\displaystyle\bigoplus_{\begin{smallmatrix}a \in Q_1,s(a)=i\end{smallmatrix}}} P_{e(a)} \ar[r]^{\ \ \ \  \ \ \ g_i}&P_{i} \ar[r] &S_i \ar[r]^{} & 0. }\] 
Then, by Proposition \cite[Theorem 3.16]{IO1},
we have a projective presentation of a simple $\Gamma$-module $S_i^\Gamma$ 
\[\tag{iii}\label{presen 1}\xymatrix@C20pt@R40pt{ \Hom_\Lambda(T_k, {\displaystyle\bigoplus_{\begin{smallmatrix}r \in R,s(r)=i\end{smallmatrix}}} P_{e(r)}) \ar[r]^{\cdot f_i\  }  &  \Hom_\Lambda(T_k,{\displaystyle\bigoplus_{\begin{smallmatrix}a \in Q_1,s(a)=i\end{smallmatrix}}} P_{e(a)}) \ar[r]^{\ \ \ \ \ \ \ \cdot g_i} & \Hom_\Lambda(T_k,P_i) \ar[r]^{} &S_i^\Gamma \ar[r]^{} & 0, }\]
and of a simple $\Gamma$-module $S_k^\Gamma$
\[\tag{iv}\label{presen 2}\xymatrix{ \Hom_\Lambda(T_k, \bigoplus_{b\in B}P_{i_b}) \ar[r]^{\cdot\alpha}  &  \Hom_\Lambda(T_k,\bigoplus_{c\in C}P_{j_c}) \ar[r]^{\cdot\beta} & \Hom_\Lambda(T_k,\tau_n^{-}P_k) \ar[r]^{} &S_k^\Gamma \ar[r]^{} & 0. }\]
Note that we have $\Hom_\Lambda(T_k, P_k)=0$ since $P_k$ is a simple projective module.
Then, by Theorem \ref{birs}, the first statement follows.
Moreover, (\ref{presen 1}),(\ref{presen 2}) are minimal projective presentations.
Indeed, we have an equivalence $\add(T_k)\cong \proj(\Gamma)$, where $\proj(\Gamma)$ denote the category of projective modules of $\mod\Gamma$.
Hence, if (\ref{presen 1}),(\ref{presen 2}) are not minimal, 
it implies that the sequence (\ref{exact seq}),(\ref{presen 0}) are not minimal, which is a contradiction.
Then by \cite{BIRS}, the second statement follows.
\end{proof}

Next, following \cite{BGP}, we define reflection functors in terms of linear representations. 
First we recall some basic results. 
Let $\Lambda=KQ/(R)$ be a finite dimensional algebra and denote by $\mathrm{rep}_K(Q,R)$ the category of finite dimensional $K$-linear
representations of $Q$ bound by $R$.
Then we recall the following fundamental theorem.
\begin{theorem}
Let $\Lambda=KQ/(R)$ be a finite dimensional algebra.
There exists a $K$-linear equivalence of categories
$$F:\mod\Lambda\to\mathrm{rep}_K(Q,R).$$
\end{theorem}

Recall that the functor $F$ associates
with any $\Lambda$-module $M$ the representation $FM=((FM)_i,\varphi_a )_{i\in Q_0,a\in Q_1}$, where $(FM)_i=M{e_i}$ for $i\in Q_0$ and the $K$-linear map $\varphi_a:M{e_i}\to M{e_j}$ for 
$a:i\to j$ in $Q_1$ is defined by $x \mapsto xa$ for $x\in Me_i$ (refer to \cite{ASS} for the precise definition). 

\begin{defi}\label{n-ref functor}
(I) Let $\Lambda=KQ/(R)$ be a finite dimensional algebra, 
$k\in Q_0$ be a sink admitting the $n$-APR tilting module $T_k$ and $(Q',R'):=\sigma_k(Q,R)$.  
We define the \emph{reflection functor}
\[ {\Q}_k^+ : \mathrm{rep}_K(Q,R)   \to  \mathrm{rep}_K(Q',R')\] as follows.
 
Take the minimal projective resolution (\ref{exact seq}) and we follow the same notations. 
For an object $X = (X_i,\varphi_a )_{}$ in $\mathrm{rep}_K(Q,R)$,
we define the object ${\Q}_k^+(X) = (X'_i , \varphi_a' )_{}$ in
$\mathrm{rep}_K(Q',R')$ as follows.
\begin{itemize}
\item[(1)]$X_i' = X_i$ for $i \neq k$, whereas $X_k'$ is the kernel of the $K$-linear map
$$(\varphi_{a_{cb}})_{a_{cb}}  :{\displaystyle\bigoplus_{\begin{smallmatrix}c\in C\end{smallmatrix}}} X_{j_c}  \to {\displaystyle\bigoplus_{\begin{smallmatrix}b\in B\end{smallmatrix}}} X_{i_b},$$
where $\varphi_{a_{cb}}$ is defined from $\varphi_{a}$ for $a\in Q_1$ naturally \cite[section III]{ASS}.

\item[(2)] $\varphi_a' = \varphi_a$ for all arrows
$a: i \to j$ in $Q_1$ with $j \neq k$, whereas, 
for $a_{c}^*:k \to {j_c}$ in $Q_1'$, define
$\varphi_{a_{c}^*} :X_k' \to X_{j_c}$ by the composition of the inclusion of $X_k'$ into
$\bigoplus_{c\in C} X_{j_c}$ with the projection onto the direct summand $ X_{j_c}$.
\end{itemize}

Next let $f = (f_i)_{i \in Q_0} :X \to Y$ be a morphism in $\mathrm{rep}_K(Q,R)$,
where $X = (X_i,\varphi_a )$ and
$Y = (Y_i,\psi_a )$.
We define the morphism
\[ {\Q}_k^+(f) = f' =(f_i')_{i \in Q_0'} :{\Q}_k^+(X) \to {\Q}_k^+(Y) \]
in $\mathrm{rep}_K(Q',R')$ as follows.
For all $i \neq k$, we let $f_i' = f_i$, whereas $f_k'$ is the unique $K$-linear map, making the following diagram commutative
\[
\xymatrix@C40pt@R40pt{
0 \ar[r] & ({\Q}_k^+(X))_k \ar[r] \ar@{.>}[d]^{f_k'} &  {\displaystyle\bigoplus_{\begin{smallmatrix}c\in C\end{smallmatrix}}}X_{j_c}
 \ar[r]^{(\varphi_{a_{cb}})_{a_{cb}}} \ar[d]^{\oplus f_{c}}   & {\displaystyle\bigoplus_{\begin{smallmatrix}b\in B\end{smallmatrix}}} X_{i_b}\ar[d]^{\oplus f_b}  \\
0 \ar[r] & ({\Q}_k^+(Y))_k \ar[r]  & {\displaystyle\bigoplus_{\begin{smallmatrix}c\in C\end{smallmatrix}}}Y_{j_c} \ar[r]^{ (\psi_{a_{cb}})_{a_{cb}}} &
{\displaystyle\bigoplus_{\begin{smallmatrix}b\in B\end{smallmatrix}}} Y_{i_b}.  }
\]
 
(II)
Let $\Gamma=KQ'/(R')$ be a finite dimensional algebra,
$k\in Q'_0$ be a source admitting the $n$-APR co-tilting module $T_k$ and $(Q,R):=\sigma_k^-(Q',R')$.
We define the \emph{reflection functor}
\[ {\Q}_k^- : \mathrm{rep}_K(Q',R')   \to  \mathrm{rep}_K(Q,R)\] as follows.

We have a minimal injective resolution of $\tau_nI_k$
\[\xymatrix{ 0\ar[r] &\tau_nI_k \ar[r] & {I^n} \ar[r]^{\gamma\ \ }&I^{n-1}\ar[r]& \ldots   \ar[r]&I^{0} \ar[r] &0.   }\]
We decompose $I^{n}$ and $I^{n-1}$ into direct sums of indecomposable modules, and denote them by
$\bigoplus_{c\in C}I_{j_c}$ and $\bigoplus_{b\in B}I_{i_b}$, respectively.
We denote the morphism $\gamma$ by 
$(\varphi_{a_{bc}})_{a_{bc}}  :\bigoplus_{c\in C}I_{j_c} \to \bigoplus_{b\in B}I_{i_b}.$  
For an object $X' = (X'_i,\varphi_a' )_{}$ in $\mathrm{rep}_K(Q',R')$,
we define the object ${\Q}_k^-(X') = (X_i , \varphi_a )_{}$ in
$\mathrm{rep}_K(Q,R)$ as follows.
\begin{itemize}
\item[(1)]$X_i = X_i'$ for $i \neq k$, whereas $X_k$ is the cokernel of the $K$-linear map
$$(\varphi_{a_{bc}})_{a_{bc}}  :{\displaystyle\bigoplus_{\begin{smallmatrix}b\in B\end{smallmatrix}}} X_{i_b}'  \to{\displaystyle\bigoplus_{\begin{smallmatrix}c\in C\end{smallmatrix}}} X_{j_c}'.$$

\item[(2)] $\varphi_a = \varphi_a'$ for all arrows
$a: i \to j$ in $Q'$ with $i \neq k$, whereas, 
for $a_{c}^*:{j_c} \to k$ in $Q_1$, define $\varphi_{a_{c}^*} :X_{j_c}' \to X_{k}$ by the composition of the inclusion of $X_{j_c}'$ into
$\bigoplus_{c\in C} X_{j_c}'$
with the cokernel projection onto $ X_{k}$.
\end{itemize}

Next let $f' = (f_i')_{i \in Q_0'} :X' \to Y'$ be a morphism in $\mathrm{rep}_K(Q',R')$,
where $X' = (X_i',\varphi_a' )$ and $Y' = (Y_i',\psi_a' )$.
We define the morphism
\[ {\Q}_k^-(f') = f =(f_i)_{i \in Q_0} :{\Q}_k^-(X')\to {\Q}_k^-(Y') \]
in $\mathrm{rep}_K(Q,R)$ as follows.
For all $i \neq k$, we let $f_i = f_i'$, whereas $f_k$ is the unique $K$-linear map, making the following diagram
commutative
\[
\xymatrix@C40pt@R40pt{
{\displaystyle\bigoplus_{\begin{smallmatrix}b\in B\end{smallmatrix}}} X_{i_b}'  \ar[r]^{(\varphi_{a_{bc}})_{a_{bc}}}  \ar[d]^{\oplus f_b} 
&{\displaystyle\bigoplus_{\begin{smallmatrix}c\in C\end{smallmatrix}}}X_{j_c}'\ar[d]^{\oplus f_{c}} \ar[r]  & ({\Q}_k^-(X'))_k \ar@{.>}[d]^{f_k}  \ar[r]&0\\
{\displaystyle\bigoplus_{\begin{smallmatrix}b\in B\end{smallmatrix}}} Y_{i_b}'  \ar[r]^{(\varphi_{a_{bc}})_{a_{bc}}}    &\ar[r]{\displaystyle\bigoplus_{\begin{smallmatrix}c\in C\end{smallmatrix}}}Y_{j_c}'
   & ({\Q}_k^-(Y'))_k  \ar[r]&0.  }
\]
\end{defi}

\begin{exam}
Let $\Lambda=KQ/(R)$ be the algebra given by Example \ref{exam ref} (2) and $X$ be a representation 
of $\mathrm{rep}_K(Q,R)$ given as follows  
$$
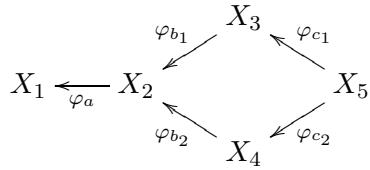
\xymatrix@C20pt@R10pt{& &X_3 \ar[dl]_{\varphi_{b_1}}    \\
X_1&X_2\ar[l]^{\varphi_{a}} &   &X_5\ar[ul]_{\varphi_{c_1}}\ar[dl]^{\varphi_{c_2}}\\
& &X_4 \ar[ul]^{\varphi_{b_2}}  }$$

Then ${\Q}_k^+(X)$ is given as follows
 
$$
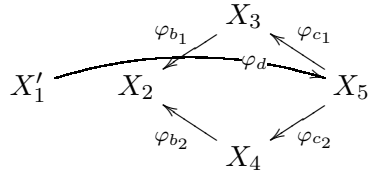
\xymatrix@C20pt@R10pt{& &X_3 \ar[dl]_{\varphi_{b_1}}  \\
X_1'\ar@/^4mm/[rrr]|(0.7){\varphi_{d}}&X_2 &   &X_5\ar[ul]_{\varphi_{c_1}}\ar[dl]^{\varphi_{c_2}}\\
& &X_4 \ar[ul]^{\varphi_{b_2}} }$$
where $X_1'$ is given by $\Ker F$ for
$\xymatrix{F:X_5 \ar[rr]^{\left(\begin{smallmatrix}\varphi_{c_1}\\ \varphi_{c_2} \end{smallmatrix}\right)}  &&X_3\oplus X_4}$.
\end{exam}
 
Then we give a relationship between ${\Q}_k^+$ (respectively, ${\Q}_k^-$) and $\T_k^+:=\Hom_{\Lambda}(T_k,-)$ (respectively, $\T_i^-:=-\otimes_{\Gamma} T_k$) as follows.
 
\begin{prop}
Let $\Lambda = KQ/(R)$ be a finite dimensional algebra and $k\in Q_0$ be a sink admitting the $n$-APR tilting module $T_k$. We denote by 
$(Q',R'):= \sigma_k(Q,R)$ and $\Gamma := KQ'/(R')$.
Then we have ${\Q}_k^+\circ F \cong F'\circ{\T}_k^+$ on $\FF_0(T_k)$ and ${\Q}_k^-\circ F' \cong F\circ{\T}_k^-$ on  $\XX_0(T_k)$,
where $F$ and $F'$ are the category equivalences, respectively.
\[
\xymatrix@C50pt@R40pt{ 
\mod KQ/(R) \ar[r]^{{\T}_k^+} \ar[d]^\cong_{F}  & \mod KQ'/(R')  \ar[d]^{\cong}_{F'} \ar[r]^{{\T}_k^-} & \mod KQ/(R)   \ar[d]^\cong_{F} \\
\mathrm{rep}_K(Q,R)\ar[r]^{{\Q}_k^+ }  & \mathrm{rep}_K(Q',R')  \ar[r]^{{\Q}_k^- } &  \mathrm{rep}_K(Q,R). }
\]
\end{prop}

\begin{proof}
We only prove that $ {\Q}_k^+\circ F \cong F'\circ{\T}_k^+ $; the proof of the second statement is similar.
Let $X$ be a $\Lambda$-module with $X\in \FF_0(T_k)$. Take a vertex $i\in Q_0$ with $i \neq k$.
By Lemma \ref{tilt theorem}, we have
\begin{eqnarray*}
(F'\Hom_\Lambda(T_k,X))_i&\cong& (\Hom_\Lambda(T_k,X))\epsilon_i\\
&\cong& \Hom_{\Gamma}({\epsilon_i}\Gamma , \Hom_\Lambda(T_k,X))\\
&\cong&\Hom_{\Gamma}(\Hom_\Lambda(T_k,e_i\Lambda) , \Hom_\Lambda(T_k,X))\\
&\cong&\Hom_{\Lambda}(e_i \Lambda , X)\\
&\cong& X{e_i}\\
&\cong&(FX)_i\\
&\cong&({\Q}_k^+(FX))_i .
\end{eqnarray*}

On the other hand, if $i = k$, we have
\begin{eqnarray*}
(F'\Hom_\Lambda(T_k,X))_k&\cong& (\Hom_\Lambda(T_k,X))\epsilon_k\\
&\cong&  \Hom_{\Gamma}({\epsilon_k}\Gamma , \Hom_\Lambda(T_k,X))\\
&\cong&\Hom_{\Gamma}(\Hom_\Lambda(T_k, \tau_n^-P_k) , \Hom_\Lambda(T_k,X))\\
&\cong&\Hom_{\Lambda}(\tau_n^-P_k , X).
\end{eqnarray*}
 
By applying $\Hom_\Lambda(-,X)$ to the minimal projective resolution (\ref{exact seq}), 
we have the following commutative diagram

\[\xymatrix@C35pt@R30pt{
0 \ar[r]  & \Hom_\Lambda (\tau_n^{-}P_k , X) \ar[r] \ar@{.>}[d] &\ar[r]^{} \ar[d]^\cong  \Hom_\Lambda
({\displaystyle\bigoplus_{\begin{smallmatrix}c\in C\end{smallmatrix}}} P_{j_c} , X)   &\Hom_\Lambda ({\displaystyle\bigoplus_{\begin{smallmatrix}b\in B\end{smallmatrix}}} P_{i_b} , X) \ar[d]^\cong \\
 0 \ar[r] & ({\Q}_k^+(FX))_k \ar[r]  & {\displaystyle\bigoplus_{\begin{smallmatrix}c\in C\end{smallmatrix}}} (FX)_{j_c}  \ar[r]^{(\varphi_{a_{cb}})_{a_{cb}}} & {\displaystyle\bigoplus_{\begin{smallmatrix}b\in B\end{smallmatrix}}} (FX)_{i_b}  }
\]
Hence we have $({\Q}_k^+ (FX))_k \cong (F'\Hom_\Lambda(T_k,X))_k$ as a $K$-vector space.
Then one can easily check that it induces an isomorphism of representations ${\Q}_k^+ F(X) \cong F'{\T}_k^+(X)$ and the isomorphism is functorial. Thus we have
${\Q}_k^+\circ F \cong F'\circ{\T}_k^+$.
\end{proof}

\end{document}